\documentclass[11pt,reqno]{amsart}
\usepackage[margin=1.25in, footskip=.5in]{geometry}
\usepackage[T1]{fontenc}
\usepackage[american]{babel}
\usepackage{amsmath,amsfonts,amssymb,amsthm,bm,wasysym}
\usepackage{mathrsfs,graphicx}
\usepackage{enumitem}
\usepackage{multirow}
\usepackage{caption,float}
\usepackage{subcaption}
\usepackage{tikz}
\usetikzlibrary{matrix,arrows,backgrounds,calc,chains,automata,positioning,patterns}
\usetikzlibrary{decorations,decorations.pathmorphing}
\usetikzlibrary{shapes.geometric,calc}

\usepackage[colorlinks,linkcolor=blue!50!black,citecolor=blue!50!black,pagebackref,hypertexnames=false, breaklinks]{hyperref}

\allowdisplaybreaks%to break long aling with newpage
\newtheorem{theorem}{Theorem}%[section]
\newtheorem{proposition}[theorem]{Proposition}
\newtheorem{corollary}[theorem]{Corollary}
\newtheorem{lemma}[theorem]{Lemma}
\theoremstyle{definition}
\newtheorem{definition}{Definition}
\newtheorem*{def*}{Definition}

\newtheorem{remark}{Remark}

\newtheorem*{thm*}{Theorem}
\newtheorem*{lem*}{Lemma}
\newtheorem*{prop*}{Proposition}
\newtheorem*{rem*}{Remark}

\numberwithin{equation}{section}
\def\R{{\mathbb R}}%Zahlkoerper
\def\Q{{\mathbb Q}}
\def\E{{\mathbb E}}
\def\P{{\mathbb P}}
\def\eps{{\varepsilon}}

\title[LQG on fractals]{Towards Liouville Quantum Gravity on fractals}
\date{Last update: \today}
\author{Patricia Alonso Ruiz}
\address{Institut f\"ur Mathematik \\ Friedrich-Schiller Universit\"at Jena}
\email{patricia.alonso.ruiz@uni-jena.de}
\author{Eva Kopfer}
\address{Institut f\"ur Angewandte Mathematik \\ Rheinische Friedrich-Wilhelms-Universit\"at Bonn}
\email{eva.kopfer@iam.uni-bonn.de}
\thanks{Research supported by an Oberwolfach Leibniz Fellowship.}
\begin{document}

\begin{abstract}
    Motivated by the recent construction of fractional Gaussian fields on the Sierpinski gasket, we study those Gaussian fields whose covariance function exhibits a logarithmic behavior. We introduce a parametric family of random measures associated with these fields, which can be regarded as the analogue to Liouville quantum gravity. The construction is based on Kahane's approach via Gaussian multiplicate chaos. We establish reflection invariance and scaling self-similarity for both the fields and the measures, and in addition construct the corresponding analogue to Liouville Brownian motion.
\end{abstract}
\maketitle
\tableofcontents
\newpage

\section{Introduction}
In mathematical physics, the Liouville quantum gravity (LQG) measure plays a central role in the study of two-dimensional quantum gravity and the Knizhnik-Polyakov-Zamolodchikov (KPZ) formula, see e.g.~\cite{knizhnik1996fractal,DS11,berestycki2016kpz,duplantier2021liouville}. From a mathematical viewpoint, the LQG measure (e.g. in $\mathbb{R}^2$) may formally be expressed as 
\begin{equation*}
    e^{\gamma h(x)}dx,
\end{equation*}
where $h$ is a Gaussian free field (GFF). This expression only holds formally since the field cannot be defined pointwise due to the logarithmic behavior of its covariance function $\mathbb{E}[h(x)h(y)]$.

\medskip

One way to construct the LQG measure consists in applying the \emph{Gaussian multiplicative chaos} (GMC) method, which goes back to Kahane~\cite{Kah85}. There, a martingale approximation was employed, assuming that the covariance kernel $\mathbb{E}[h(x)h(y)]$ is decomposable into a sum of non-negative kernels. This condition is hard to check in practice, which lead to further investigations by Rhodes and Vargas in~\cite{RV14}, where they approximated the fields via convolution with appropriate mollifiers. This approach yields a random measure that is unique in law, but not as a function of the underlying field. 
A related construction using circle averages was proposed by Duplantier and Sheffield in~\cite{DS11} for the special case of a Gaussian free field in two dimensions. An alternative approach based on thick points of the underlying field can be found in~\cite{Ber17}. 

\medskip

To a certain extent and from a theoretical viewpoint, one may speak of an LQG measure in any metric measure space, in particular beyond the Euclidean setting, as soon as there is a meaningful concept of log-correlated GFF. Constructions of such fields have been investigated for the torus in~\cite{david2016liouville}, in the Riemannian setting in~\cite{dSHKS24}, and also explored in the abstract framework of Dirichlet form theory in~\cite{FO20}. 
In the present paper, we rely on a characterization of GMC through a transformation due to Shamov~\cite{Sha16}, which had already been pointed out in~\cite{ST94}. The recent construction of GFFs in the context of fractals, and specifically on the Sierpinski gasket in~\cite{BL22,BC23}, has sparkled the interest in investigating the associated LQG measure here.

\medskip

From a more geometric viewpoint, the LQG measure can heuristically be viewed as the random volume measure associated with the formal random metric $e^{\gamma h(x)}|dx|^2$ of Liouville quantum gravity. Garban-Rhodes-Vargas introduced in~\cite{GRV16} the \emph{Liouville Brownian motion} (LBM) in $\mathbb{R}^2$ as the natural diffusion associated with this random geometry, which was independently considered by Berestycki in~\cite{Ber17}. Subsequently, alternative constructions using Dirichlet forms have been proposed in~\cite{Shi19,FO20} and in~\cite{dSHKS24} in the context of Riemannian manifolds. In the present paper, we follow the general approach via Dirichlet form theory to define this process as a time change of the Brownian motion on the Sierpinski gasket.

\medskip

The paper is organized as follows: After introducing in Section~\ref{S:background} notations and basic geometric and analytic results concerning the Sierpinski gasket, Section~\ref{S:GFF_on_SG} deals with the construction and (invariance) properties of the random field which we interpret as the analogue of the planar GFF. That field is used in Section~\ref{S:LQGM} to construct the associated LQG measures by approximation, see Theorem~\ref{T: LQGM}. It is worth pointing out that these measures are indexed by a parameter $\gamma>0$ that lies in a certain, non-optimal, regime, discussed in Remark~\ref{R:gamma_range}. Finally, Section~\ref{S:LBM} is devoted to the construction of the LBM associated with a given LQG measure on the gasket, see Theorem~\ref{T:def_LBM_SG}. The result relies on the key property of the LQG measure being $\mathbb{P}$-a.s. smooth.

%%%%------------------------------------------------------
\section{Background and notation}\label{S:background}
\subsection{The Sierpinski gasket}\label{SS:SG_basics}
In this section we review notation and basic results concerning the Sierpinski gasket that are of relevance to the purpose of the paper. For further details we refer the reader e.g. to~\cite[Chapter 3]{Str06}.

\medskip

The standard Sierpinski gasket $K$, see Figure~\ref{F:SG}, is the unique compact subset of $\mathbb{R}^2$ that satisfies the fixed-point equation
\begin{equation*}%\label{E:SG_fixed_point}
	K=\bigcup_{i=1}^3F_i(K),
\end{equation*} 
where $F_i\colon\mathbb{R}^2\to\mathbb{R}^2$ are the mappings
\begin{equation*}
	F_i(p):=\frac{1}{2}(p-p_i)+p_i
\end{equation*}
and $V_0:=\{p_1,p_2,p_3\}$ denotes the set of vertices of an equilateral triangle of side length one. The set $V_0$ is also regarded as the natural boundary of $K$. %The self-similarity property~\eqref{E:SG_fixed_point} has its roots in the \emph{Banach fixed point theorem}, c.f.~\cite[Section 3.1]{Hut81}.
As a metric measure space, $K$ is equipped with the Euclidean metric and the standard self-similar Bernoulli measure $\mu$ that gives the same weight to each triangular cell of the same side-length. This measure thus satisfies
\begin{equation*}%\label{E:def_measure_SG}
	\mu(F_{w}(V_0))=\frac{1}{3^m}
\end{equation*}
for any $F_{w}:=F_{w_1}{\circ}\cdots \circ F_{w_m}$, where $w=w_1\ldots w_m\in\{1,2,3\}^m$, $m\geq 0$. The latter measure is comparable to the
normalized $d_H$-dimensional Hausdorff measure, where $d_H=\frac{\log 3}{\log 2}$ is the (Euclidean) Hausdorff dimension of $K$. For simplicity, and since $K$ is compact, we set both its diameter and measure to be $1$.

\begin{figure}[H]
	\includegraphics[scale=.15]{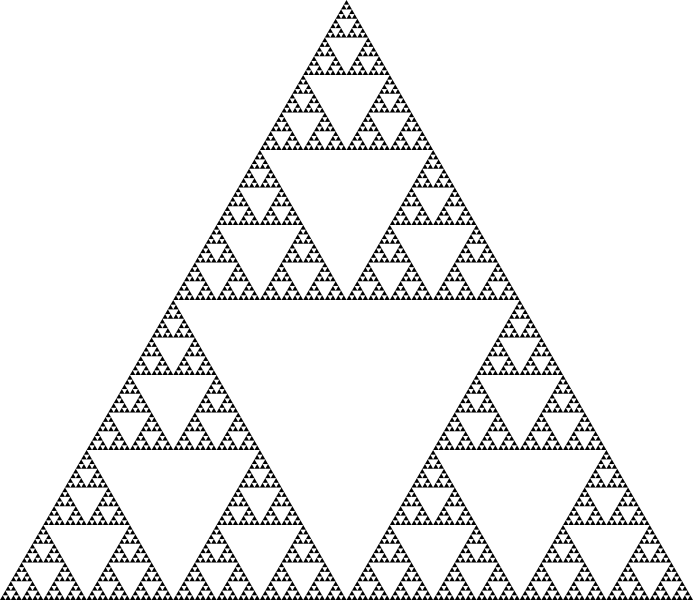}
	\caption{The standard Sierpinski gasket.}
	\label{F:SG}
\end{figure}

Regarding functions on $K$, we will denote by $L^2(K):=L^2(K,\mu)$ the space of square-integrable functions with respect to the measure $\mu$, and by $\langle u,v\rangle$ the associated standard $L^2$-inner product with norm $\|u\|^2:=\langle u,u\rangle$.

\medskip

The Dirichlet Laplacian on $K$ is the self-adjoint operator on $L^2(K)$ that is the infinitesimal generator of the local and regular Dirichlet form $(\mathcal{E},\mathcal{F})$ given by
\begin{equation}\label{E:def_DF_SG}
    \begin{cases}
	\mathcal E(u,v):=\lim\limits_{n\to\infty}\left(\frac 53\right)^n\sum\limits_{x,y\in V_n,x\sim y}(u(x)-u(y))(v(x)-v(y)),\\
    \mathcal{F}:=\{u\in C(K)\colon \mathcal{E}(u,u)<\infty\}.
    \end{cases}
\end{equation}
%defined on functions $u,v$ vanishing on the boundary $\partial K$.
Its associated heat semigroup $\{P_t\}_{t\geq 0}$
% , given by
% \begin{equation}\label{E:Semigroup}
% 	P_tu(x)=e^{t\Delta}u(x)=\int_Kp_t(x,y)u(y)\,d\mu(y),% =\sum_{\varphi}e^{\lambda_\varphi t}\langle u|\varphi\rangle_{L^2}.
% \end{equation}
has a heat kernel $p_t(x,y)$ that satisfies the sub-Gaussian estimates 
\begin{equation}\label{E:subG_HKE}
	p_t(x,y)
	\asymp c_1 t^{-d_S/2} \exp \bigg( - c_2\Big(\frac{d(x,y)^{d_W}}{t}\Big)^{\frac1{d_W-1}} \bigg)\tag{HKE}
\end{equation}
for $\mu$-a.e. $x,y\in K$ and $t\in (0,1)$. %$t\in (0,({\rm diam}\,K)^{d_W})$. 
Here, $d_S$ and $d_W$ denote respectively the so-called spectral and walk dimension, which in the case of the Sierpinski gasket are
 \begin{equation*}%\label{E:dimensions_SG}
 	d_W=\frac{\log 5}{\log 2}\qquad\text{and}\qquad d_S=\frac{\log 9}{\log 5}=\frac{2d_H}{d_W},
 \end{equation*}      
see e.g.~\cite{Kum93}. In addition, it is also know that the heat kernel satisfies the global estimate
\begin{equation}\label{E:HKE_large_t}
 	p_t(x,y)\leq c_3t^{-d_S/2}\exp(-t/2)
\end{equation}
for $\mu$-a.e. $x,y\in K$ and $t>1$. We refer the reader to~\cite{Kum93} and~\cite[Section 5]{Bar98} for further details.
%
%The heat kernel estimate~\eqref{E:subG_HKE} in particular implies that $P_t$ is ultracontractive, see e.g.~\cite[Theorem 2.1.5]{Dav90}, whence $-\Delta$ has a pure point spectrum with only accumulation point at infinity~\cite[Theorem 2.1.4]{Dav90}. 
%Note that all eigenvalues $\lambda$ are non-negative. %because $-\Delta$ is a non-negative operator. 

\medskip

The infinitesimal generator $\Delta$ associated with $\{P_t\}_{t\geq 0}$ is a non-positive operator with pure point spectrum that is regarded as the standard intrinsic Laplacian on $K$. 
%Let $0\leq \lambda_{1}\leq \lambda_{2}\leq\lambda_{3}\leq\ldots$ denote the ordered eigenvalues including multiplicity of $-\Delta$ with zero boundary condition. %the standard (Dirichlet) Laplace operator  on $K$.
The eigenvalues of $-\Delta$ can be described by means of a genealogy tree, where each distinct eigenvalue has a unique ``genealogical line''. %As reference to the first ancestor of an eigenvalue $\lambda$, one speaks of an $i$-series eigenvalue, where $i\in\{2,5,6\}$. 
%The lowest $5$-series and $6$-series Dirichlet eigenvalue have generation of birth $1$ respectively $2$, and we will denote them by $\lambda_0^{(5)}$ and $\lambda_1^{(6)}$. 

\medskip

We next mention some relevant features of the spectrum that will be used later on and refer the interested reader e.g. to~\cite[Section 3.3]{Str06} for further details. Below, $0\leq \lambda_{1}\leq \lambda_{2}\leq\lambda_{3}\leq\ldots$ denote the ordered eigenvalues including multiplicity, and $\lambda_1^{(6)}$ the sixth lowest Dirichlet eigenvalue.

\begin{proposition}\label{P:6series_props}
	For any $j\geq 1$, define the numbers
    \begin{equation}\label{E:def_6series_ev_mult}
        N_j:=\frac{1}{2}(3^{j+1}-3)\qquad\text{and}\qquad\Lambda_j:=5^{j-2}\lambda_1^{(6)}.
    \end{equation}
    Then, for any $j\geq 2$,
	\begin{enumerate}[wide=0em,itemsep=.5em,label={\rm(\roman*)}]
		%\item $\Lambda_j$ is the lowest 6-series eigenvalue with generation of birth $j$.
		\item $\lambda_{N_j}=\Lambda_j$, i.e. ${\rm mult}\,\Lambda_j=N_{j}$.
		\item $\#\{k\geq 1\colon \varphi_k\text{ eigenfunction of }-\Delta\text{ with eigenvalue }\lambda_k\leq \Lambda_j,\}=N_{j+1}$.
	\end{enumerate}
\end{proposition}

%\textbf{Notation.} Throughout the paper, 
%we will use $\langle u,v\rangle$ for to denote the standard $L^2$-product of $u,v\in L^2(K)$, and $\langle u|v\rangle$ for the $L^2$-product of functions in different spaces, e.g. $u\in H^{-s}(K)$ and $v\in H^s(K)$. \pat{(Check is is what is meant)}\eva{This should be deleted since it is not correct and misplaced. See Section \ref{SS:function_spaces_SG} below.}
%we write $A\simeq B$ meaning that the quantity $A$ is bounded above and below by different constants times the quantity $B$. 

%%%%%------------------------------------------------------
%%%%%------------------------------------------------------

\subsection{Riesz potentials}
%%%%%%%----------------------------
%%%%%%%----------------------------
By virtue of the spectral theorem, the heat semigroup can be expressed as
\begin{equation*}%\label{E:Semigroup_spectral}
    P_tu(x)=\sum_{k= 1}^\infty e^{-\lambda_{k}t}\langle u,\varphi_k\rangle\varphi_k(x).
\end{equation*}
% The Bessel and Riesz potentials admit equivalent expressions in the $L^2$-sense in terms of the spectral resolution and of the heat semigroup. Namely, for any $s>0$,
% \begin{equation}\label{E:def_Bessel}
%     \begin{aligned}
%     (I-\Delta)^{-s/2}u
%     &=\int_0^\infty (1+\lambda)^{s/2}dE_\lambda(u)=\sum_{\varphi}(1+\lambda_\varphi)^{s/2}\langle u|\varphi\rangle\varphi\\
%     &=\frac{1}{\Gamma(s/2)}\int_0^\infty t^{s/2}e^{-t}P_tu\,\frac{dt}{t}
%     \end{aligned}
% \end{equation}
% for the Bessel potential, and
%\eva{das steht schon oben}where $\langle \cdot,\cdot\rangle$ denotes the standard inner product in $L^2(K)$, 
The Riesz potential admits an equivalent expression in the $L^2$-sense in terms of the spectral resolution of the heat semigroup, that is
\begin{equation}\label{E:def_Riesz}
    \begin{aligned}
    (-\Delta)^{-s/2}u
    &=\int_0^\infty \lambda^{-s/2}dE_\lambda(u)=\sum_{k=1}^\infty\lambda_k^{-s/2}\langle u,\varphi_k\rangle\varphi_k
    =\frac{1}{\Gamma(s/2)}\int_0^\infty t^{s/2}P_tu\,\frac{dt}{t},
    \end{aligned}
\end{equation}
c.f.~\cite[Proposition 2.4]{HZ05}. 
% \begin{remark}
%     To avoid problems \pat{(for now, not sure how much we need to go into this)} we will consider the Laplace operator with Dirichlet boundary conditions and focus on the Riesz potential.
% \end{remark}
The above can be regarded as an integral operator by writing in~\eqref{E:def_Riesz} the heat semigroup in terms of the heat kernel. 
%we obtain , for $u$ defined pointwise,
% \begin{align*}
%     (-\Delta)^{-s/2}u(x)
%     &=\frac{1}{\Gamma(s/2)}\int_0^\infty t^{s/2}\int_K u(y)p_t(x,y)\,d\mu(y)\frac{dt}{t}\\
%     &=\int_K u(y)\frac{1}{\Gamma(s/2)}\int_0^\infty t^{s/2}p_t(x,y)\,\frac{dt}{t}d\mu(y).
% \end{align*}
%In this way 
In this way, we arrive at the definition of the Riesz kernel of order $s>0$.

\begin{definition}%\label{D:def_Riesz_kernel}
Let $s>0$. The fractional Riesz kernel of order $s$ is given by
\begin{equation}\label{E:def_Riesz_kernel}
    k_{s}(x,y):=\frac{1}{\Gamma(s/2)}\int_0^\infty t^{s/2}p_t(x,y)\,\frac{dt}{t}
\end{equation}
for all $x,y\in K$, $x\neq y$.

\end{definition}

\begin{remark}
Note that, a priori, $k_s(x,y)$ need not be well-defined for all $s>0$. When $0<s<d_S$, it is known that the Riesz kernel~\eqref{E:def_Riesz_kernel} satisfies
\begin{equation*}%\label{E:Riesz_kernel_bounds}
    C^{-1}d(x,y)^{-\frac{d_W}{2}(d_S-s)}\leq k_s(x,y)\leq Cd(x,y)^{-\frac{d_W}{2}(d_S-s)},
\end{equation*}
see~\cite[Proposition 2.7, Proposition 2.8]{HZ05}.
\end{remark}
% Moreover, Weyl's law~\eqref{E:Weyl_law} implies that the series~\eqref{E:Riesz_kernel_spectral} converges pointwise as long as $0<s<d_S$ \pat{(we can make this rigorous later)} and so the question is now

% \begin{center}
%     what happens at criticality, i.e. when $s=d_S$?
% \end{center}
At criticality, i.e. when $s=d_S$, the kernel displays a similar logarithmic behavior as in the Euclidean/Riemannian manifold case when $s$ equals the dimension of the space. The proof of this estimate can be found in~\cite[Proposition 2.8]{BC23}, and we include it here  in detail for completeness.
\begin{proposition}\label{P:log_correlation_kernel}~\cite[Proposition 2.8]{BC23}
There exists constants $C_0,c_0>0$ such that
    \begin{equation}\label{E:log_correlation_kernel}
      c_0\log\frac{1}{d(x,y)} \leq k_{d_S}(x,y)\leq  C_0\Big(1+\log\frac{1}{d(x,y)}\Big).
    \end{equation}
\end{proposition}
\begin{proof}
    Setting $s=d_S$ in~\eqref{E:def_Riesz_kernel} we first separate
    \begin{equation*}
        k_{d_S}(x,y)
        =\frac{1}{\Gamma(d_S/2)}\int_0^{1} t^{d_S/2}p_t(x,y)\,\frac{dt}{t}+\frac{1}{\Gamma(d_S/2)}\int_{1}^\infty t^{d_S/2}p_t(x,y)\,\frac{dt}{t}
        =I_1+I_2.
    \end{equation*}
    Applying the heat kernel estimates~\eqref{E:subG_HKE} and the change of variables $v=\frac{d(x,y)}{t^{1/d_W}}$,
    \begin{align}
        I_1
        &\leq c_1\int_0^{1}\exp \bigg( - c_2\Big(\frac{d(x,y)}{t^{1/d_W}}\Big)^{\frac{d_W}{d_W-1}} \bigg)\,\frac{dt}{t}\notag\\
        &=d_Wc_1\int_{d(x,y)}^\infty\exp\Big(\!\!-c_2v^{\frac{d_W}{d_W-1}}\Big)\frac{dv}{v}\notag\\
        &=d_Wc_1\int_{d(x,y)}^1\exp\Big(\!\!-c_2v^{\frac{d_W}{d_W-1}}\Big)\frac{dv}{v}
        +d_Wc_1\int_1^\infty \exp\Big(\!\!-c_2v^{\frac{d_W}{d_W-1}}\Big)\frac{dv}{v}\notag\\
        &=:I_{1,1}+I_{1,2}.\label{E:Riesz_kernel_bounds_01}
    \end{align}
    The first integral above is bounded by
    \begin{equation*}
        I_{1,1}\leq d_Wc_1\int_{d(x,y)}^1\frac{dv}{v}=d_Wc_1\Big[\log v\Big]_{d(x,y)}^1=-d_Wc_1\log d(x,y)=d_Wc_1\log\frac{1}{d(x,y)}
    \end{equation*}
    and the second by a constant. %\pat{(it should be possible to compute more precisely but not sure how helpful)}. 
    % \begin{align*}
    %     I_{1,2}=d_Wc_1c_3
    %     &=\frac{1}{\Gamma(s/2)}\int_0^\infty t^{s/2}p_t(x,y)\,\frac{dt}{t}
    % \end{align*}
    For the second term in~\eqref{E:Riesz_kernel_bounds_01} one applies the estimate~\eqref{E:HKE_large_t} for long times to get
    \begin{align*}
        I_{1,2}
        &\leq c_3\int_1^\infty e^{-t/2}\frac{dt}{t}=:c_4<\infty
    \end{align*}
    In total we arrive at
    \begin{equation*}
        k_{d_S}(x,y)\leq d_Wc_1\log\frac{1}{d(x,y)}+c_5.
    \end{equation*}
    For the lower bound, we apply the corresponding lower bound of the heat kernel~\eqref{E:subG_HKE} and argue as in~\eqref{E:Riesz_kernel_bounds_01} to get
    \begin{align*}
        k_{d_S}(x,y)
        &\geq C_1\int_0^{1}\exp \bigg( - C_2\Big(\frac{d(x,y)}{t^{1/d_W}}\Big)^{\frac{d_W}{d_W-1}} \bigg)\,\frac{dt}{t}\notag\\
        &=d_WC_1\int_{d(x,y)}^\infty\exp\Big(\!\!-C_2v^{\frac{d_W}{d_W-1}}\Big)\frac{dv}{v}\notag\\
        &\geq d_WC_1\int_{d(x,y)}^1\exp\Big(\!\!-C_2v^{\frac{d_W}{d_W-1}}\Big)\frac{dv}{v}\\
        &\geq d_WC_1e^{-C_2}\int_{d(x,y)}^1\frac{dv}{v}=d_WC_1e^{-C_2}\log\frac{1}{d(x,y)}
    \end{align*}
    as desired.
\end{proof}

Particularly useful for the construction of Gaussian free fields on fractals in Section~\ref{S:GFF_on_SG} is the following kernel expansion stated in~\cite[Proposition 2.4]{HZ05}.
\begin{lemma}\label{L:Riesz_spectral_kernel}
    For any $s>d_S/2$, the fractional Riesz kernel~\eqref{E:def_Riesz_kernel} admits the expansion
    \begin{equation*}%\label{E:Riesz_kernel_spectral}
        k_s(x,y)=\sum_{k= 1}^\infty\lambda_{k}^{-s/2}\varphi_k(x)\varphi_k(y),\quad x,y\in K,
    \end{equation*}
   where the convergence of the series is understood in the $L^2(K\times K,\mu\otimes\mu)$-sense. 
\end{lemma}
\begin{proof}
%Note that by the spectral theorem for the heat kernel
For any $\ell\geq 1$, define
\begin{equation*}
k_{s,\ell}(x,y):%=\frac1{\Gamma(s/2)}\int_0^\infty t^{s/2-1}\sum_{j=1}^\ell e^{-\lambda_j t}\varphi_i(x)\varphi_j(y)\, dt
=\sum_{j=1}^\ell\lambda_j^{-s/2}\varphi_j(x)\varphi_j(y).
\end{equation*}
The orthogonality of the eigenfunctions $\varphi_j$ together with Proposition~\ref{P:6series_props} yield
    \begin{equation*}
    \begin{aligned}
        \|k_{s,\ell}(x,y)\|^2_{L^2(K\times K, \mu\otimes\mu)}
        &=\sum_{j=1}^\infty\lambda_j^{-s}
        =\sum_{j=1}^{\infty}\sum_{k=N_{j-1}+1}^{N_j}\lambda_{k}^{-s}\\
        &\simeq \sum_{j=1}^{\infty}(N_j-N_{j-1})\Lambda_{j-1}^{-s}
        = \sum_{j=1}^{\infty}3^j\Lambda_{j-1}^{-s}
        \simeq \sum_{j=1}^{\infty}3^j5^{-sj}
    \end{aligned}
    \end{equation*}
    which is bounded for $s>\frac{d_S}{2}=\frac{\log 3}{\log 5}$. %and therefore the sequence $\{k_{s,\ell}(x,y)\}_{\ell\geq 1}$ converges in $L^2(K\times K,\mu\otimes\mu)$. 
% Hence, 
%     \begin{align*}
%    \|k_{s,\ell}(x,y)- k_s(x,y)\|_{L^2(K\times K,\mu\otimes\mu)}\xrightarrow{\ell\to\infty}0.
%     \end{align*}
Here, $A\simeq B$ means that the quantity $A$ is bounded above and below by different constants times the quantity $B$. 
\end{proof}
%%%----------------------------------------------------------------
\subsection{Function spaces}%\label{SS:function_spaces_SG}
The function spaces playing a role in the current analysis are those associated with the Riesz potentials. In the sequel we follow its characterization from~\cite[Definition 3.1]{HZ05}.

\medskip

\begin{definition}%\label{D:Sobolev_spaces}
    Let $s>0$. The Sobolev-Riesz potential space of order $s$ is defined as
    \begin{align*}
        %&H^s:=\{v=(I-\Delta)^{s/2}u\text{ for some }u\in L^2(K,\mu)\}%\label{E:def_Bessel_Sobolev_sp}\\
        &H^s(K):=\{v=(-\Delta)^{-s/2}u\text{ for some }u\in L^2(K)\}%\label{E:def_Riesz_Sobolev_sp}
    \end{align*}
    with associated norm given by
    \begin{equation*}%\label{E:def_seminorms}
        \|v\|_{H^s}:=\|(-\Delta)^{s/2}v\|.
    \end{equation*}
\end{definition}
By virtue of the spectral decomposition, the norm %in Definition~\ref{D:Sobolev_spaces} 
above admits the expression  \begin{equation}\label{E:spectral_seminorm}
        %\|v\|_{H^s}^2=\sum_\varphi (1+\lambda_\varphi)^{-s}|\langle u|\varphi\rangle|^2
        %\qquad\text{and}\qquad
        \|v\|_{H^s}^2=\sum_{k=1}^\infty \lambda_k^{s}\langle v,\varphi_k\rangle^2.  
    \end{equation}
\begin{remark}\label{R:negative_Hs}
    We define the space $H^{-s}(K)$ as the completion of $L^2(K)$ with respect to the norm
    \begin{equation*}%\label{E:spectral_neg_seminorm}
        \|v\|_{H^{-s}}^2=\|(-\Delta)^{-s/2}v\|^2=\sum_{k=1}^\infty \lambda_k^{-s}\langle v,\varphi_k\rangle^2,
    \end{equation*}
so that, formally, $H^{-s}(K)=(-\Delta)^{s/2}[L^2(K)]$. Since
\begin{equation*}
    |\langle u,v\rangle|\leq \|u\|_{H^{-s}}\|v\|_{H^s}
\end{equation*}
for all $u\in L^2(K)$ and $v\in H^s(K)$, the inner product in $L^2$ continuously extends to a bilinear form
\begin{align*}
   \langle \cdot,\cdot\rangle_{H^{-s},H^s} \colon H^{-s}(K)\times H^{s}(K)\to\R.
\end{align*}
In particular, $H^{-s}(K)$ can be identified with the dual space of $H^s(K)$. We simply write $\langle u,v\rangle$ for the extended dual pairing between $H^{-s}(K)$ and $H^s(K)$, whenever $u\in H^{-s}(K)$ and $v\in H^s(K)$. 
\end{remark}

Further, following~\cite{BC23}
we consider the space of test functions 
 \begin{align*}
 		\mathcal S(K):=\Big\{u\in C_0(K)\colon \lim_{n\to\infty}n^k\Big|\int_K\varphi_n(y)u(y)\, d\mu(y)\Big|=0\quad\forall k\geq0\Big\}
 	\end{align*}
 	endowed with the topology induced by the family of norms $\{\|(-\Delta)^ku\|\}_{k\geq 0}$. We denote its dual by $\mathcal S'(K)$ and use the same notation $\langle \cdot,\cdot\rangle$ for the dual pairing between $\mathcal S(K)$ and $\mathcal S'(K)$. Note that $\mathcal S(K)$ embeds continuously into $H^s(K)$ for all $s>0$.

%%%%%-------------------------------------------------------
%%%%%------------------------------------------------------

%%%%%------------------------------------------------------
%%%%%------------------------------------------------------
\section{Fractal Gaussian free fields}\label{S:GFF_on_SG}
The aim of this section is to construct and analyze a random field $h$ on $K$ which may be viewed as the analogue to the planar Gaussian free field on the Sierpinski gasket. The law of this field is formally given by the ``probability measure''
\begin{align*}
{\rm Law}_{h}(du)=\frac1{Z}\exp\Big(-\frac12 \|(-\Delta)^{d_S/4}u\|_{L^2}^2\Big)\, du
\end{align*}
for a suitable (non-existing) normalization constant $Z$. 
 Note that, due to the logarithmic divergence of the Riesz kernel $k_{d_S}$ associated with $(-\Delta)^{-d_S/2}$, the field $h$ cannot be defined pointwise, but only in a distributional sense.

\begin{definition}%\label{def: gf}
We say that	a centered Gaussian random element $h\colon\Omega\to \mathcal S'(K)$ is a \emph{fractal Gaussian free field} (FGFF) on $K$  if its covariance is given by
	 \begin{equation}\label{E:cov_FGFF}
		 \mathbb E[\langle h,u\rangle \langle h,v\rangle]=\int_K\int_{K} k_{d_S}(x,y)u(x)v(y)\, d\mu(x)\, d\mu(y)
     \end{equation}
     for all $u,v\in \mathcal S(K)$.
\end{definition}

 \begin{theorem}
 	The FGFF on $K$ exists, and it is unique in law.
 \end{theorem}
%\begin{proof}
 %	Following \cite[Theorem 2.12]{BC23}, we first define
 	%\begin{align*}
 		%S(K):=\Big\{u\in C_0(K)\colon \lim_{n\to\infty}n^k\Big|\int_K\varphi_n(y)u(y)\, d\mu(y)\Big|=0\quad\forall k\geq0\Big\}
 	%\end{align*}
 	%endowed with the topology induced by the family of norms $\|(-\Delta)^ku\|_{L^2}$. This defines a Fr\'echet nuclear space. Second, note that the functional $\chi\colon S(K)\to  [0,\infty)$ given by
 	%\begin{equation*}
 		%\chi\colon u\to \exp\Big(-\frac{1}{2}\|u\|_{H^{d_S/2}}^2\Big)
        %\exp(-\frac12\int_K|(-\Delta)^{-d_S/4}u|^2\, d\mu)
 	%\end{equation*}
 	%is continuous at $0$ and positive definite \cite[Theorem 2.11]{BC23}. Thus, the Bochner-Minlos theorem, see e.g.~\cite[Theorem 4.3, p.410]{VTC87} implies the existence of a characteristic functional taking values in $S'(K)$ that uniquely characterizes a centered Gaussian random element in $S'(K)$ with covariance~\eqref{E:cov_FGFF}. Since $H^\eps(K)$ embeds continuously into $S(K)$ for every $\eps>0$ \pat{(or is it $H^{-\eps}$??)}\eva{ no it is $H^\eps$ but it is not clear and maybe it is not what I want to say}, the claim follows.
 %\end{proof}

 An abstract proof of the existence and uniqueness of $h$ via the Bochner-Minlos theorem can be found in~\cite[Theorem 2.12]{BC23}. %\pat{(We include the proof for completeness - or not??) }\eva{We show a little bit more since we prove existence in the negative Sobolev space.}
It is also possible to provide a more constructive proof of the existence of the FGFF by means of a suitable approximating sequence of Gaussian fields, c.f.~\cite[Remark 2.4]{BC23}. We analyze this approximation in the next section, which will further imply that $h\in H^{-\eps}(K)$ $\P$-a.s. for every $\eps>0$.
%%%%%------------------------------------------------------
\subsection{Approximating Gaussian fields}
As in~\cite[Definition 2.16]{BC23}, let $\{\xi_k\}_{k\geq 1}$ be a sequence of i.i.d. standard normal distributed random variables. For each $\ell\geq 1$, define the random field
\begin{equation}\label{E:def_approx_RF_SG}
    h_{\ell}(x):=\sum_{k=1}^{N_\ell}\xi_{k}\lambda_{k}^{-\frac{d_S}{4}}\varphi_k(x), \qquad x\in K,
\end{equation}
where $N_\ell$ is the multiplicity of the eigenvalue $\Lambda_\ell$, see Proposition~\ref{P:6series_props}. The covariance function of this field is thus given by
\begin{equation}\label{E:def_cov_approx_RF_SG}
  k_{d_S,\ell}(x,y):=  \mathbb{E}[h_{\ell}(x)h_{\ell}(y)]=\sum_{k=1}^{N_\ell}\lambda_{k}^{-\frac{d_S}{2}}\varphi_k(x)\varphi_k(y),\quad x,y\in K.
\end{equation}
% \begin{remark}\label{R:index_reason}
% The covariance function $k_\ell(x,y)$ can be regarded as the spectral projection of $k_{d_S}(x,y)$ to the eigenspace $E_{\leq \Lambda_\ell}$ of eigenfunctions with eigenvalue $\lambda\leq \Lambda_m$, i.e.
% \[
% h_\ell(x)=\sum_{\varphi\in E_{\leq \Lambda_{\ell}}}\xi_\varphi \lambda_{\varphi}^{-\frac{d_S}{4}}\varphi(x)
% =\sum_{\lambda\leq\Lambda_\ell}\sum_{\varphi\in E_\lambda}\xi_\varphi \lambda_{\varphi}^{-\frac{d_S}{4}}\varphi(x),\qquad x\in K.
% \]
% \end{remark}
The convergence of the sequence $\{h_{\ell}\}_{\ell\geq 1}$ can be shown using the spectral properties of $-\Delta$ that were recorded in Proposition~\ref{P:6series_props}.
\begin{theorem}\label{T:convergence}
\begin{enumerate}[wide=0em,label={\rm (\roman*)},itemsep=.5em]
\item%\label{P:approx_L2_Hs}
 The sequence $\{h_{\ell}\}_{\ell\geq 1}$ converges both in $L^2(\Omega)$ and $\mathbb{P}$-a.s. to the FGFF $h$ on $K$. In particular,
 \begin{equation*}%\label{E:h_as_lim}
    h\stackrel{d}{=}\sum_{k=1}^\infty\xi_{k}\lambda_{k}^{-\frac{d_S}{4}}\varphi_k
 \end{equation*}
and, $\mathbb{P}$-a.s., it holds that $h\in {H}^{-\varepsilon}(K)$ for any $\varepsilon>0$ and $h\notin L^2(K)$.
\item%\label{P:L2_weak_approx}
    For every $u\in{H}^{-d_S/2}(K)$, the sequence $\{\langle h_{\ell},u\rangle\}_{\ell\geq 1}$ is a centered, $L^2$-bounded martingale with respect to the filtration $\{\mathcal F_\ell\}_{\ell\geq1}$, where
    \begin{equation*}
        \mathcal{F}_{\ell}:=\sigma(\{\xi_1,\ldots,\xi_{N_\ell}\}).
    \end{equation*}
    Moreover, the sequence converges in $L^2(\Omega)$ and $\mathbb{P}$-a.s. to a centered Gaussian random variable, denoted by $h\cdot u$. In particular, $\P$-a.s.
    \begin{equation*}
    h\cdot u=\sum_{k=1}^{\infty}\xi_{k}\lambda_{k}^{-\frac{d_S}{4}}\langle \varphi_k,u\rangle%\in L^2(\Omega)
    \end{equation*}
    and
    \begin{equation*}
        \E[(h\cdot u)^2]=\|u\|_{H^{-d_S/2}}^2.
    \end{equation*}
   
\end{enumerate}
\end{theorem}
\begin{proof}
\begin{enumerate}[wide=0em,label={\rm (\roman*)},itemsep=.5em]
    \item  Fix $\varepsilon>0$. We show first that $h\in H^{-\varepsilon}(K)$ $\mathbb{P}$-a.s. By definition, applying~\eqref{E:spectral_seminorm} and Proposition~\ref{P:6series_props} we have
    \begin{equation*}
        \begin{aligned}
        \mathbb{E}\Big[\|h\|_{{H}^{-\varepsilon}}^2\Big]
        &=\sum_{k=1}^{\infty}\lambda_{k}^{-\frac{d_S}{2}-\varepsilon}\mathbb{E}[\xi_k^2]
        =\sum_{j=1}^{\infty}\sum_{k=N_{j-1}+1}^{N_j}\lambda_{k}^{-\frac{d_S}{2}-\varepsilon}
        \simeq \sum_{j=1}^{\infty}(N_j-N_{j-1})\Lambda_{j-1}^{-\frac{d_S}{2}-2\varepsilon}.
        \end{aligned}
    \end{equation*}
    Once again, $A\simeq B$ means that the quantity $A$ is bounded above and below by different constants times the quantity $B$. 
    Substituting in the above the expression for $N_j$ from~\eqref{E:def_6series_ev_mult},
    \begin{align}
        \mathbb{E}\big[\|h\|_{{H}^{-\varepsilon}}^2\big]
        &=\sum_{j=1}^\infty \Big(\frac{1}{2}(3^{j+1}-3)-\frac{1}{2}(3^j-3)\Big)\Lambda_{j-1}^{-\frac{d_S}{2}-\varepsilon}\notag\\
        &=\sum_{j=1}^\infty 3^j\Lambda_{j-1}^{-\frac{d_S}{2}-\varepsilon}
        %=\sum_{j=1}^\infty 3^j\big(5^{j-3}\lambda_1^{(6)} \big)^{-\frac{d_S}{2}-\varepsilon}\notag\\
        %&=(5^{-3}\lambda_1^{(6)})^{-\frac{d_S}{2}-\varepsilon}\sum_{j=1}^\infty \big(3{\cdot}5^{\frac{\log 3}{\log 5}}5^{\varepsilon}\big)^{-j}\notag\\
        %&=(5^{-3}\lambda_1^{(6)})^{-\frac{d_S}{2}-\varepsilon}\sum_{j=1}^\infty 5^{-j\varepsilon}
        \simeq \sum_{j=1}^\infty 5^{-j\varepsilon}
        ,\label{E:approx_L2_Hs_01}
    \end{align}
    where the last line follows from Proposition~\ref{P:6series_props}. The series above is finite if and only if $\varepsilon>0$, and the particular case $\varepsilon=0$ yields $h\notin L^2(K)$ $\mathbb{P}$-a.s. 

    \medskip

    %We now prove convergence in $L^2(\Omega,H^{-\varepsilon})$.
    Similarly, in view of~\eqref{E:approx_L2_Hs_01} and~\eqref{E:def_6series_ev_mult}, for any $\ell\geq 1$ it holds that
    \begin{equation*}%\label{E:approx_L2_Hs_02}
        \mathbb{E}\big[\|h-h_{\ell}\|_{{H}^{-\varepsilon}}^2\big]\apprle\sum_{j=N_\ell+1}^\infty 5^{-j\varepsilon}\xrightarrow{\ell\to\infty}0
    \end{equation*}
    and thus $\{h_{\ell}\}_{\ell\geq 1}$ converges in $L^2(\Omega)$. By virtue of Tschebychev's inequality and~\eqref{E:approx_L2_Hs_01}, for any $\delta>0$
    \begin{equation*}
        \sum_{\ell=1}^\infty\mathbb{P}(\|h-h_{\ell}\|_{{H}^{-\varepsilon}}>\delta) 
        \leq\frac{1}{\delta^2}\sum_{\ell=1}^\infty \mathbb{E}\Big[\|h-h_{\ell}\|_{{H}^{-\varepsilon}}^2\Big]<\infty.
    \end{equation*}
    The fast convergence theorem, see e.g.~\cite[Theorem 6.12]{Kle14}, implies $\mathbb{P}$-a.s. convergence and hence convergence in distribution. 

    \medskip

    Finally, we verify that $h$ is in fact a FGFF: Since $h$ is an $L^2(\Omega)$-limit of centered Gaussian random variables, it is a centered Gaussian random variable itself with covariance
    \begin{equation*}
        \begin{aligned}
        \mathbb{E}\big[\langle  h,u\rangle \langle  h, v\rangle\big]
        &=\sum_{k_1=1}^{\infty}\sum_{k_2=1}^{\infty}\mathbb{E}\Big[\xi_{{k_1}}\xi_{{k_2}}(\lambda_{{k_1}}\lambda_{{k_2}})^{-\frac{d_S}{4}}\langle \varphi_{k_1},u\rangle\langle \varphi_{k_2},v\rangle\Big]\\
        &=\sum_{k=1}^{\infty}\lambda_{k}^{-\frac{d_S}{2}}\langle \varphi_k,u\rangle\langle \varphi_k,v\rangle\\
        &=\int_{K}\int_{K}u(x)v(y)\sum_{k=1}^{\infty}\lambda_{k}^{-\frac{d_S}{2}}\varphi_k(x)\varphi_k(y)\,d\mu(y)d\mu(x)\\
        &=\int_{K}\int_{K}u(x)v(y)k_{d_S}(x,y)\,d\mu(y)\,d\mu(x),
        \end{aligned}
        \end{equation*}
        for any $u,v\in H^{d_S/2}(K)$. Note that in the first equality we used the convergence in $L^2(\Omega)$ of $h_{\ell}$ to $h$, the independence of $\xi_{k}$ in the second, and Lemma~\ref{L:Riesz_spectral_kernel} in the last two lines.
    
    \item Let $u\in{H}^{-d_S/2}(K)$. By definition of $h_\ell$, for any $\ell\geq 1$
    \begin{equation*}
        \langle h_{\ell}\,,\,u\rangle=\sum_{k=1}^{N_\ell}\xi_{k}\lambda_{k}^{-\frac{d_S}{4}}\langle \varphi_k,u\rangle,
    \end{equation*}
    which is a sum of independent centered (Gaussian) random variables and thus a martingale with respect to the natural filtration $\mathcal{F}_{\ell}$. Moreover, independence implies that
    \begin{equation*}
        \begin{aligned}
            \sup_{\ell\in\mathbb{N}}\mathbb{E}[|\langle h_{\ell},u\rangle|^2]
            &=\sup_{\ell\in\mathbb{N}}\mathbb{E}\bigg[\Big|\sum_{k=1}^{N_\ell}\xi_{k}\lambda_{k}^{-\frac{d_S}{4}}\langle \varphi_k,u\rangle\Big|^2\bigg]\\
            &=\sup_{\ell\in\mathbb{N}}\sum_{k=1}^{N_\ell}\mathbb{E}[\xi_{k}^2]\,\lambda_{k}^{-\frac{d_S}{2}}|\langle \varphi_k,u\rangle|^2\\
            %&=\sup_{\ell\in\mathbb{N}}\sum_{k=1}^{N_\ell}\lambda_{k}^{-\frac{d_S}{2}}|\langle \varphi_k,u\rangle_{L^2(K)}|^2\\
            &=\sum_{k=1}^{\infty}\lambda_{k}^{-\frac{d_S}{2}}|\langle \varphi_k,u\rangle|^2
            =\|u\|_{{H}^{-d_S/2}}^2<\infty.
        \end{aligned}
    \end{equation*}
    Therefore, $\{\langle h_\ell,u\rangle\}_{\ell\geq 1}$ is an $L^2(\Omega)$-bounded martingale. By virtue of Doob's $L^p$-convergence theorem for martingales, see e.g.~\cite[Theorem 11.10]{Kle14}, the sequence converges $\mathbb{P}$-a.s. and in $L^2(\Omega)$ to a random variable denoted by $h\cdot u\in L^2(\Omega)$. Moreover,
    \begin{align*}
         h\cdot u= \sum_{k=1}^{\infty}\xi_{k}\lambda_{k}^{-\frac{d_S}{4}}\langle \varphi_k,u\rangle\qquad \mathbb{P}\text{-a.s.}% \text{ in }L^2(\Omega,\mathbb{P})\text{ and }
    \end{align*}

    Since each $\langle h_\ell,u\rangle$ defines a centered Gaussian random variable, the $L^2$-convergence of $\langle h_\ell,u\rangle$ to $h\cdot u$ implies that $h\cdot u$ is a centered Gaussian random variable with variance given by $\|u\|_{H^{-d_s/2}}^2$.
    \end{enumerate}
%     {\color{gray}
%      In addition, since
%     \begin{align*}
%         \mathbb{E}[|\langle h_{K,\infty}\rangle_{L^2(K)}-\langle h_{K,\ell}\rangle_{L^2(K)}|^2]
%         &=\mathbb{E}[|\langle u|h_{K,\infty}-h_{K,\ell}\rangle_{L^2(K,\mu)}|^2]\\
%         &=\mathbb{E}\bigg[\Big|\sum_{k=N_\ell+1}^{\infty}\xi_{k}^2]\lambda_{k}^{-\frac{d_S}{2}}|\langle \varphi_k,u\rangle_{L^2(K,\mu)}\Big|^2\bigg]\\
%        % &=\sum_{k=N_\ell+1}^{\infty}\mathbb{E}[\xi_{\varphi_k}^2]\lambda_{\varphi_k}^{-\frac{d_S}{2}}|\langle \varphi_k,u\rangle_{L^2(K,\mu)}|^2\\
%         &=\sum_{k=N_\ell+1}^{\infty}\lambda_{k}^{-\frac{d_S}{2}}\big|\langle \varphi_k,u\rangle_{L^2(K,\mu)}\big|^2\xrightarrow{\ell\to\infty}0,
%     \end{align*}
%     the limit equals $\langle h_{K,\infty},u\rangle_{L^2(K,\mu)}$.
% }
\end{proof}

\begin{remark}
Analogous arguments as those in the proof of Theorem~\ref{T:convergence}(ii) yield that the FGFF $h$ can always be approximated by the sequence of Gaussian random fields 
\begin{equation*}
    h_{\ell}(x):=\sum_{k=1}^{N_\ell}\tilde{\xi}_k\lambda_{k}^{-\frac{d_S}{4}}\varphi_k(x), \qquad x\in K,
\end{equation*}
with $\tilde{\xi}_k:=\langle h, \varphi_k\rangle$. 
\end{remark}

%We now show that the limit $h_{K,\infty}$ from~\eqref{E:h_as_lim} is a random distribution in $H^{d_S/2}$.

%The random distribution obtained gives the quantum Gaussian field from Definition~\ref{D:def_QGF} on the Sierpinski gasket $K$.
%\begin{corollary}\label{C:h_in_cap}
  %  The limit $h_{K,\infty}$ from~\eqref{E:h_as_lim} is a random distribution in $H^{d_S/2}(K)$ and $h_K=h_{K,\infty}$ is the quantum Gaussian field on $K$ associated with its intrinsic standard Dirichlet Laplacian.
%\end{corollary}
We now discuss how the FGFF interplays with the symmetries and scaling properties of $K$.

%%%%%------------------------------------------------------
\subsection{Invariance properties}
%%%%%------------------------------------------------------
%%%%%------------------------------------------------------
%%%%%------------------------------------------------------

\subsubsection{Invariance by symmetries}\label{SS:symmetries_SG}
Let $\sigma_i\colon\mathbb{R}^2\to\mathbb{R}^2$, $i=1,2,3$, denote each of the reflections about the axes that divide the Sierpinski gasket into half, see Figure~\ref{F:reflections} . 

\begin{figure}[H]
	\includegraphics[scale=.35]{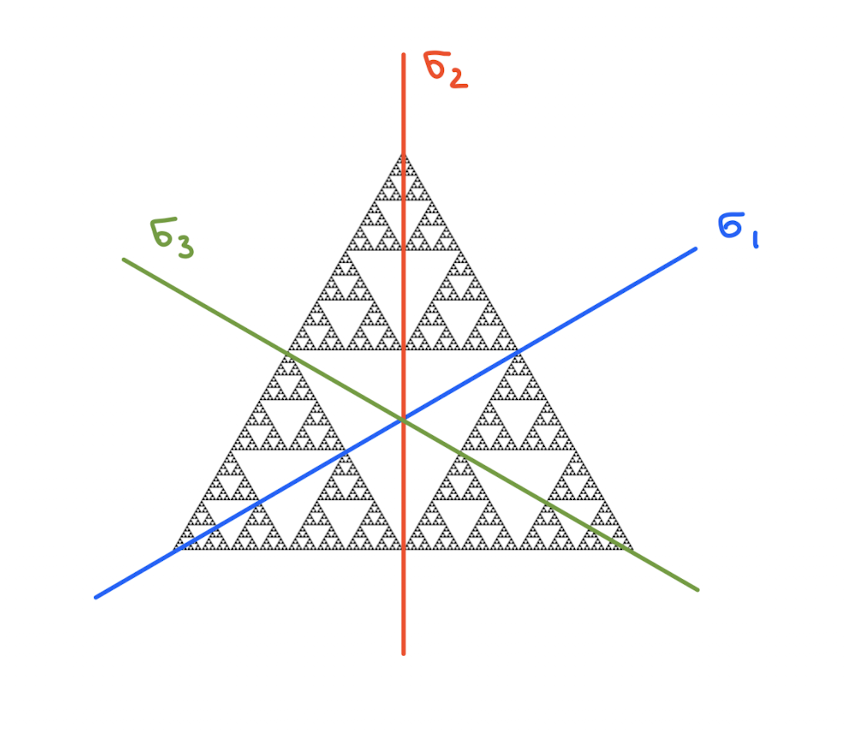}
	\caption{The reflections $\sigma_1$, $\sigma_2$ and $\sigma_3$.}
	\label{F:reflections}
\end{figure}

Then, $\sigma_i(V_0)=V_0$ and 
\begin{align*}
    \mathcal E(u\circ\sigma_i,v\circ\sigma_i)=\mathcal E(u,v).
\end{align*}
In particular, as in the proof of~\cite[Proposition 3.9]{BL22}, for any $i=1,2,3$ and $s>0$ it holds that
    \begin{equation}\label{E:kernel_symmetries}
    k_s(\sigma_i(x),\sigma_i(y))=k_s(x,y).
    \end{equation}
Since the underlying measure $\mu$ is also invariant under the symmetries $\sigma_i$, the invariance of the FGFF on $K$ follows.

\begin{theorem}\label{T:invariance_symmetry}
    For any $i=1,2,3$ it holds that
    \begin{equation}\label{E:invariance_symmetry}
        h\stackrel{d}{=}h{\circ}(\sigma_i)^{-1},
    \end{equation}
    where $h{\circ}(\sigma_i)^{-1}$ is understood in the distributional sense, i.e. $\langle h\circ(\sigma_i)^{-1},u\rangle=\langle h,u\circ \sigma_i\rangle$.
\end{theorem}
\begin{proof}
    By construction, both $h$ and $h{\circ}(\sigma_i)^{-1}$ are centered Gaussian fields. Moreover, by the invariance~\eqref{E:kernel_symmetries} of the kernel $k_{d_S}$ and that of the measure $\mu$, they have the same covariance function, which implies~\eqref{E:invariance_symmetry}.
\end{proof}
\subsubsection{Self-similarity by scaling}%\label{SS:scaling_SG}
We prove now the self-similarity of the FGFF on $K$ under the contractions $F_w$ from Section~\ref{SS:SG_basics}, where $w=w_1w_2\ldots w_n\in\{1,2,3\}^n$, $n\geq 1$. For a fixed word $w$, let $K_w:=F_w(K)\subset K$ denote the rescaled Sierpinski gasket associated with the contraction mapping $F_w$. %, see Figure~\ref{F:rescaled_SG}. 
The space is naturally equipped with the rescaled metric $d_{w}$ and measure $\mu_{w}$ that satisfy
    \begin{equation}\label{eq: measureinv}
    \begin{aligned}
        d_w(x,y)&=2^{-n}d(F_w^{-1}(x),F_w^{-1}(y)),\\
        (F_w^{-1})_\#\mu_{w}&=3^{-n}\mu.
    \end{aligned}
    \end{equation}
In particular, for any $u,v\in L^2(K_w,\mu_{w})$
    \begin{equation*}%\label{eq: l2inv}
    \langle u,v\rangle_{L^2(K_w,\mu_{w})}=3^{-n}\langle u\circ F_w,v\circ F_w\rangle_{L^2(K,\mu)}.
    \end{equation*}

% \begin{figure}[H]
%      \includegraphics[scale=.15]{SG.png}
%     \caption{The rescaled Sierpinski gasket associated with the map $F_{12}$.}
%     \label{F:rescaled_SG}
% \end{figure}
%\eva{The Dirichlet Laplacian on $K_w$ is understood with Dirichlet boundary condition on $V_w=F_w(V_0)$. 
The standard Dirichlet form with Dirichlet boundary conditions $(\mathcal{E}_w,\mathcal{F}_w)$ on $L^2(K_w,\mu_{w})$ satisfies
\begin{equation*}
     \mathcal E_w(u,v)=\left(\frac53\right)^n\mathcal E(u\circ F_w,v\circ F_w)
\end{equation*}
and note that $u|_{F_w(V_0)}\equiv 0$ for all $u\in \mathcal{F}_{w}$. As in the proof of \cite[Proposition 3.11]{BL22}, the associated heat kernel satisfies
\begin{equation}\label{E:scaling_HK_SG}
    p_t^w(x,y)=3^n p_{5^n t}(F_w^{-1}(x),F_w^{-1}(y)) 
    \end{equation}
for all $x,y\in K_w$, see e.g.~\cite[Section 3.4.2]{BL22}.

\medskip

To state the self-similarity under $F_w$ properly, we would like to consider a random field $h^{w}$ on $K_w$ formally given by
    \begin{equation*}
    h^{w}=h{\circ}F_w^{-1}.
    \end{equation*}
Note however, that the above equality has to be read in a distributional sense because $h$ is not defined pointwise. Therefore, analogously to Remark~\ref{R:negative_Hs}, by regarding $H^{-\eps}(K_w)$ as the completion of $L^2(K_w,\mu_{w})$ with respect to the corresponding norm $\|\cdot\|_{H^{-\eps}(K_w)}$ we set
    \begin{equation}\label{E:def_weak_prod_Kw}
    \langle h^{w},u\rangle_{H^{-\varepsilon}(K_w),H^\varepsilon(K_w)}:=3^{-n}\langle h,u\circ F_w\rangle_{H^{-\varepsilon}(K),H^\varepsilon(K)}
    \end{equation}
for any $u\in H^{\varepsilon}(K_w)$. 
In this way we may now state and prove the corresponding self-similarity of the FGFF $h$. %For the ease of reading, we will write in the proof simply $\langle\cdot,\cdot\rangle$ for the inner product~\eqref{E:def_weak_prod_Kw}.
\begin{theorem}\label{T:scaling_invariance}
	Let $h$ be the FGFF on $K$. For any $w\in\{1,2,3\}^n$ and $n\geq 1$,
	\begin{align*}
		h^{w}{:=}h\circ F_w^{-1},
	\end{align*}
    defined in the distributional sense \eqref{E:def_weak_prod_Kw},
	is the FGFF on $K_w$.
\end{theorem}
\begin{proof}
	Note first that, by definition, $h^{w}$ is a centered Gaussian field on $K_w$. 
	In view of~\eqref{E:scaling_HK_SG} and the definition of the Riesz kernel~\eqref{E:def_Riesz_kernel} it follows that
	\begin{equation*}
		k_s^w(F_w(x),F_w(y))
		=\frac{3^n}{5^{ns/2}\Gamma(s/2)}\int_0^\infty t^{s/2-1}p_t(x,y)\, dt
		=\frac{3^n}{5^{ns/2}}k_s(x,y)
	\end{equation*}
	for $\mu$-a.e. $x,y\in K$. In particular, in the critical case $s=d_S$ the latter reads
	\begin{equation}\label{eq: kernelinv}
		k_{d_S}^w(F_w(x),F_w(y))=k_{d_S}(x,y).
	\end{equation}
	In addition, %by definition of $h_{K_w}$ and $h_K$ it holds that
    for any $u,v\in H^\varepsilon(K_w)$
	\begin{equation*}
    \begin{aligned}
		\mathbb{E}[\langle h^w,u\rangle \langle h^w,v\rangle]
		&=3^{-2n}\mathbb{E}[\langle h,u\circ F_w\rangle \langle h,v\circ F_w\rangle]\\
		&=3^{-2n}\int_K\int_{K} k_{d_S}(x,y)u(F_w(x))v(F_w(y))\, d\mu(x)\, d\mu(y).
	\end{aligned}
    \end{equation*}
    By performing a change of variables, using~\eqref{eq: kernelinv} and~\eqref{eq: measureinv}, we finally obtain
    \begin{equation*}
	\begin{aligned}
		\mathbb{E}[\langle h^w,u\rangle \langle h^w,v\rangle]
		&=3^{-2n}\int_K\int_K k^w_{d_S}(F_w(x),F_w(y))u(F_w(x))v(F_w(y))\, d\mu(x)\, d\mu(y) \\
		&=\int_{K_w}\int_{K_w} k^w_{d_S}(\tilde x,\tilde y)u(\tilde x)v(\tilde y)\, d\mu_w(\tilde x)\, d\mu_w(\tilde y).
	\end{aligned}
    \end{equation*}
\end{proof}

% {\color{gray}
% The latter motivates now the following general definition.
% \begin{definition}
%     Let $h_X$ be fractal Gaussian free field on a metric measure space $(X,d_X,\mu_X)$ and let $\Phi\colon X\to X$ be a mapping. The notation $h_X{\circ}\Phi^{-1}$ means that
%     \begin{equation}\label{E:def_weak_composition}
%         \langle h_X{\circ}\Phi^{-1},u\rangle_{H^{-\varepsilon}(\Phi(X)),H^{\varepsilon}(\Phi(X))}
%         =\frac{d\mu{\circ}\Phi}{d\mu}\langle h_X,u{\circ}\Phi\rangle_{H^{-\varepsilon}(X),H^{\varepsilon}(X)}
%     \end{equation}
%     for any $u\in H^{\varepsilon}(\Phi(X))$.
% \end{definition}
% }

%%%%%------------------------------------------------------
%%%%%------------------------------------------------------
\section{Liouville Quantum Gravity measure}\label{S:LQGM}
The aim of this section is to construct a random measure that takes into account the geometry induced by the FGFF $h$ on $K$. Formally, such a measure is given by
\begin{equation*}
	e^{\gamma h(x)}\, d\mu(x)
\end{equation*}
for some suitable $\gamma\in\mathbb R$. However, since $h$ is almost surely not a function, we need to replace $h$ by appropriately rescaled and well-defined approximations.% $h_{K,\ell}$. 
% Moreover, a rescaling is needed in order to get convergence, which is why we consider the limit of the random measures
% \begin{equation*}
% 	d\nu_{\ell,\gamma}(x)=e^{\gamma h_{K,\ell}(x)-\frac{\gamma^2}2\mathbb E[h_{K,\ell}^2(x)]}\, d\mu(x).
% \end{equation*}
%This method goes back to Kahane \cite{Kah85} and is called \emph{Gaussian multiplicative chaos}, in short (GMC).
\medskip

In the special case of the Gaussian free field in $\mathbb{R}^2$, the resulting random measure is called \emph{Liouville Quantum Gravity measure} (LQG measure). The concept appeared in theoretical physics in the early 1980s, see~\cite{Pol81}, and its rigorous probabilistic construction has been treated afterwards by several authors~\cite{Kah85,DS11,RV14,Ber17,Sha16}.  

\medskip
%We follow the approach by Shamov \cite{Sha16}, who gave a first canonical definition of GMC that does not depend on any specific approximation scheme by employing a simple transformation rule.

\begin{remark}\label{R:gamma_range}
    Throughout this section, the parameter $\gamma$ is chosen in the range 
   \begin{equation*}
      |\gamma|< \sqrt{d_H/C_0},
  \end{equation*}
    where $C_0>0$ is the constant from~\eqref{E:log_correlation_kernel}.
\end{remark}
%%%%------------------------------------------------------------
%%%%------------------------------------------------------------
\subsection{Approximations}
In this section we analyze the approximating measures that arise from the eigenfunction approximation $h_{\ell}$ with covariance kernels $k_{d_s,\ell}$ introduced in~\eqref{E:def_approx_RF_SG} and~\eqref{E:def_cov_approx_RF_SG} respectively. 
%\begin{definition}\label{D:def_approx_random_measure}
%	Let $\gamma>0$ and $\ell\geq 1$. For any random function $h\colon\Omega\to\bigcap_{\varepsilon>0}H^\varepsilon(K)$ with the same covariance structure as $h_{K,\ell}$ we define the random measure
For fixed $\gamma>0$ and $\ell\geq 1$, these measures are defined as
	\begin{equation}\label{E:def_approx_LQM}
		d\nu_{{\ell}}^\gamma(x):=e^{\gamma h_{\ell}(x)-\frac{\gamma^2}{2}\mathbb{E}[h_{\ell}(x)^2]}d\mu(x).
	\end{equation}
%\end{definition}
%Note that $h_{K,\ell}$ is by construction a (centered) Gaussian field, whence the distribution of $h$ is fully determined by its covariance structure.
%and they are in particular subcritical GMCs.
Note that $h_{\ell}$ is by construction a (centered) Gaussian field, whence the distribution of its is fully determined by its covariance structure.
In particular, we show that $\nu_{\ell}^\gamma(K)$ is a non-negative $L^2(\Omega)$-bounded martingale for $|\gamma|<\sqrt{d_H/C_0}$.

\begin{proposition}\label{P:measure_bdd_martingale}
    Let $|\gamma|<\sqrt{d_H/C_0}$. For any $u\in C_b(K)$, the sequence $\{\nu_{\ell}^\gamma(u)\}_{\ell\geq 1}$ given by
    \begin{equation*}%\label{E:measure_bdd_martingale}
        \nu_{\ell}^\gamma(u):=\int_K u\,d\nu_{\ell}^\gamma,\qquad \ell\geq 1,
    \end{equation*}
    is an $L^2(\Omega)$-bounded martingale with respect to the filtration $\mathcal{F}=\{\mathcal{F}_\ell\}_{\ell\geq 1}$, where
    \begin{equation*}
        \mathcal{F}_\ell:=\sigma(h_\ell),\qquad\ell\geq 1.
    \end{equation*}
\end{proposition}
  
\begin{proof}
    Let $u\in C_b(K)$. We first show that $\nu_{\ell}^\gamma(u)\in L^1(\Omega)$ for all $\ell\geq 1$: Since $h_{\ell}(x)$ is a centered Gaussian random variable with variance $k_{d_S,\ell}(x,x)$, %i.e. $h_{K,\ell}(x)\sim N(0,\mathbb{E}[h_{K,\ell}^2(x)])$ and
    its moment generating function %(Laplace transform)  
    equals
    \begin{equation}\label{E:mgf_approx_h_K}
        \mathbb{E}[e^{\gamma h_{\ell}(x)}]=e^{\frac{\gamma^2}{2}\mathbb{E}[h_{\ell}^2(x)]}=e^{\frac{\gamma^2}{2}k_{d_S,\ell}(x,x)},
    \end{equation} 
    where the last equality follows from~\eqref{E:def_cov_approx_RF_SG}. Thus,
    \begin{equation*}%\label{E:expectation_approx_measure}
    \mathbb E[|\nu_{\ell}^\gamma(u)|]\leq \int_K |u(x)|\mathbb{E}[e^{\gamma h_{\ell}(x)}]e^{-\frac{\gamma^2}2k_{d_S,\ell}(x,x)}\, d\mu(x)=\int_K|u|\, d\mu<\infty.
    \end{equation*}
    To show the martingale property, we rewrite
    \begin{align*}
    d\nu_{\ell+1}^\gamma(u)(x)=Y_{\ell,\ell+1}(x)\, d\nu_{\ell}^\gamma(u)(x),
    \end{align*}
    with 
    \begin{equation*}
    Y_{\ell,\ell+1}(x)=e^{\gamma(h_{\ell+1}-h_{\ell})(x)-\frac{\gamma^2}2(k_{d_S,\ell+1}-k_{d_S,\ell})(x,x)}.
    \end{equation*}
    Note now that
    \begin{align*}
    \mathbb E[Y_{\ell,\ell+1}(x)\,|\,\mathcal{F}_\ell]=\mathbb E[Y_{\ell,\ell+1}(x)]=1,
    \end{align*}
    where the last equality follows analogously to~\eqref{E:mgf_approx_h_K}. Thus, 
    \begin{equation*}%\label{E:Martingale_prop}
        \begin{aligned}
        \mathbb E[\nu_{\ell+1}^\gamma(u)\,|\,\nu_{{\ell}}^\gamma(u)]
        &=\int_K u(x)\,\mathbb E[Y_{\ell,\ell+1}(x)\,|\,\nu_{\ell}^\gamma(u)]\, d\nu_{\ell}^\gamma(u)(x)\\
        &=\int_K u(x)\, d\nu_{\ell}^\gamma(u)(x)=\nu_{\ell}^\gamma(u),
        \end{aligned}
    \end{equation*}
    whence $\{\nu_{\ell}^\gamma(u)\}_{\ell\geq 1}$ is a martingale.

    \medskip

    To show the $L^2(\Omega)$-boundedness, we prove that
    \begin{equation*}%\label{E:Doobs_requires}
        \sup_{\ell\geq 1}\mathbb{E}[\nu_{\ell}^\gamma(u)^2]<\infty.
    \end{equation*}
    Applying Fubini,~\eqref{E:cov_FGFF} and~\eqref{E:mgf_approx_h_K}, we obtain
    \begin{equation*}
        \begin{aligned}
            \mathbb{E}[\nu_{\ell}^\gamma(u)^2]
            &=\mathbb{E}\bigg[\Big(\int_K u\,d\nu_{\ell}^\gamma\Big)^2\bigg]\\
            &=\mathbb{E}\bigg[\int_K\int_Ku(x)u(y)e^{\gamma\big(h_{\ell}(x)+h_{\ell}(y)\big)-\frac{\gamma^2}{2}\mathbb{E}[h_{\ell}(x)^2+h_{\ell}(y)^2]}d\mu(y)\,d\mu(x)\bigg]\\
            &=\int_K\int_Ku(x)u(y)e^{-\frac{\gamma^2}{2}\mathbb{E}[h_{\ell}(x)^2+h_{\ell}(y)^2]}\mathbb{E}\Big[e^{\gamma\big(h_{\ell}(x)+h_{\ell}(y)\big)}\Big]d\mu(y)\,d\mu(x)\\
            &=\int_K\int_Ku(x)u(y)e^{-\frac{\gamma^2}{2}(k_{d_s,\ell}(x,x)+k_{d_s,\ell}(y,y))}\mathbb{E}\Big[e^{\gamma\big(h_{\ell}(x)+h_{\ell}(y)\big)}\Big]d\mu(y)\,d\mu(x)\\
            &=\int_K\int_Ku(x)u(y)e^{-\frac{\gamma^2}{2}(k_{d_s,\ell}(x,x)+k_{d_s,\ell}(y,y))+\frac{\gamma^2}{2}\big(k_{d_s,\ell}(x,x)+k_{d_s,\ell}(y,y)+2k_{d_s,\ell}(x,y)\big)}\Big]d\mu(y)\,d\mu(x)\\
            &=\int_K\int_Ku(x)u(y)e^{\gamma^2k_{d_s,\ell}(x,y)}d\mu(y)\,d\mu(x).
        \end{aligned}
    \end{equation*}
    % where the second last equality is due to the fact that $h_\ell(x)\sim N(0,\lambda_k^{-d_S/2}\varphi_k^2(x))$ and thus
    % \begin{equation*}
    %     \mathbb{E}\Big[e^{\gamma(h_\ell(x)+h_\ell(y))}\Big]
    %     =e^{\frac{\gamma^2}{2}\big(k_\ell(x,x)+k_\ell(y,y)+2k_\ell(x,y)\big)}.
    % \end{equation*}
    Since $u$ is bounded and $K$ is compact, the above estimate, Proposition~\ref{P:log_correlation_kernel} and Lemma~\ref{lemma: integrability} below finally yield
    \begin{equation*}%\label{E:L2_bounded_approx}
        \begin{aligned}
            \mathbb{E}\bigg[\Big(\int_K u\,d\nu_{\ell}^\gamma\Big)^2\bigg]
            &\leq \|u\|_{L^\infty}^2\int_K\int_K e^{\gamma^2|k_{d_s,\ell}(x,y)|}d\mu(y)\,d\mu(x)\\
            &\leq \|u\|_{L^\infty}^2\int_K\int_K e^{\gamma^2C(1-\log d(x,y))}d\mu(y)\,d\mu(x)\\
            &= \|u\|_{L^\infty}^2e^{\gamma^2C}\int_K\int_K\frac{1}{d(x,y)^{\gamma^2C}}d\mu(y)\,d\mu(x)\\
            &\leq \tilde{C} \|u\|_{L^\infty}^2.
        \end{aligned}
    \end{equation*}
    %which is uniformly bounded by virtue of since $|\gamma|<\sqrt{d_H/C}$.
\end{proof}

The following estimate used in the previous proof appears in~\cite{BC23} and its proof is included here for completeness.
\begin{lemma}\label{lemma: integrability}
	Let $|\gamma|<\sqrt{d_H/C_0}$. Then,
	\begin{equation*}
		\int_K\int_K\frac{1}{d(x,y)^{\gamma^2C}}d\mu(y)\,d\mu(x)<\infty.
	\end{equation*}
\end{lemma}
%\pat{(Don't think this is an iff)}
\begin{proof}
	For any $r>0$ and $x\in K$, let $B(x,r)$ denote the ball of radius $r$ centered at $x$. Using a standard integral decomposition argument and the $d_H$-Ahlfors regularity of $\mu$,
	\begin{align*}
		\int_K\int_K\frac{1}{d(x,y)^{\gamma^2C_0}}d\mu(y)\,d\mu(x)
		&\leq \int_K\sum_{k=0}^\infty\int_{B(x,2^{-k}R_K){\setminus}B(x,2^{-k-1}R_K)}\frac{1}{d(x,y)^{\gamma^2C_0}}d\mu(y)\,d\mu(x)\\
		&\leq \int_K\sum_{k=0}^\infty\int_{B(x,2^{-k}){\setminus}B(x,2^{-k-1})}(2^{k+1})^{\gamma^2C_0}d\mu(y)\,d\mu(x)\\
		&\leq \int_K\sum_{k=0}^\infty(2^{k+1})^{\gamma^2C_0}\mu(B(x,2^{-k}))\,d\mu(x)\\
		&\apprle\int_K\sum_{k=0}^\infty(2^{k+1})^{\gamma^2C_0}(2^{-k})^{d_H}d\mu(x)\\
		&\apprle \mu(K)2^{\gamma^2C}\sum_{k=0}^\infty 2^{-k(d_H-\gamma^2C_0)}.
	\end{align*}
	The latter series converges if and only if $d_H-\gamma^2C_0>0$, which is the case since by assumption $|\gamma|<\sqrt{d_H/C_0}$.
\end{proof}

%%%%--------------------------------------------------------------
%%%%--------------------------------------------------------------
\subsection{Liouville Quantum Gravity measure}\label{SS:LQGM}

%As already mentioned in the introduction, the classical approach to construct LQG measure via Gaussian multiplicative chaos due Kahane~\cite{Kah85} requires the covariance kernel of the underlying field to be decomposable into a sum of non-negative kernels. This condition is hard to check in practice.
% \medskip
% Rhodes and Vargas used in~\cite{RV14} approximating fields defined by convolution with appropriate mollifiers. The obtained random measure is unique in law, but not as a function of the underlying field. 
% A related construction using circle averages was proposed by Duplantier and Sheffield~\cite{DS11} in the special case of a Gaussian free field in two dimensions. An alternative approach based on thick points of the underlying field can be found in~\cite{Ber17}.

%\medskip
Next, we give the definition of the Liouville quantum gravity measure.
Here, we use the characterization through the transformation rule \eqref{eq: trans} established in~\cite{Sha16}. %and had already been pointed out in~\cite{ST94}.

%\eva{In Konsistenz zu der Defintion FGFF im folgenden auch eine Korrektur.}
\begin{definition}%\label{D:def_LQGM_on_K}
  We call a random measure $\nu^\gamma$ on $K$ a \emph{Liouville Quantum Gravity measure} (LQG measure) if
   \begin{equation}\label{eq: trans}
   \mathbb E\left[\int_K F(h,x)\, d\nu^\gamma(x)\right]=\mathbb{E}\left[\int_KF(h+\gamma k_{d_S}(x,\cdot),x)\, d\mu(x)\right]
   \end{equation}
   for every nonnegative measurable functional $F$ on $H^{-\varepsilon}(K)\times K$.
\end{definition}

\begin{theorem}\label{T: LQGM}
    For any $|\gamma|<\sqrt{d_H/C_0}$, there exists a unique LQG measure $\nu^\gamma$ on $K$ such that \begin{equation}\label{E:approx_expected_measure}
    \mathbb{E}[\nu^\gamma(E)]=\mu(E),
    \end{equation}
    for any $E\subseteq K$. In addition, $\nu_\ell^\gamma$ converges weakly to $\nu^\gamma$ in the sense of Borel measures on $K$ in $L^2(\Omega)$, i.e.
    \begin{equation*}%\label{E:measure_L1_conv}
    \nu^\gamma_{\ell}(u)\xrightarrow[\ell\to\infty]{L^2(\Omega)}{\nu}^\gamma(u)
    \end{equation*}
    for all $u\in C_b(K)$.  
\end{theorem}

\begin{proof}
The uniqueness follows immediately from \cite[Corollary 5]{Sha16}. 

   %The uniqueness directly follows from~\eqref{eq: trans}. Indeed, if there is one, then the RHS of~\eqref{eq: trans} does not depend on $\nu^\gamma$ and it is a well-defined number for any $F$. Thus, the LHS characterizes a probability measure on $K\times H^{-\varepsilon}(K)$ and in particular it characterizes the probability measure $\frac{1}{\nu^\gamma(K)}\nu^\gamma$ on $K$ by disintegration. %Finally, the volume $\nu_\gamma(K)$ is characterized by taking $F(h,x)=F(h)$.

   \medskip
   
   The existence follows by approximation: By Proposition~\ref{P:measure_bdd_martingale}, the sequence $\{\nu_{\ell}^\gamma(u)\}_{\ell\geq 1}$ is an $L^2$-bounded martingale for any $u\in C_b(K)$. The $L^p$-convergence theorem for martingales (with $p=2$), see e.g.~\cite[Theorem 11.4]{Kle14}, guarantees the existence of a limit in $L^2(\Omega)$. The convergence in the sense of Borel measures follows then by the Riesz-Markov-Kakutani representation theorem.
   %and $\mathbb{P}$-a.s. which we call $\nu^\gamma(u)$. It follows that $\{\nu_{\ell}^\gamma(u)\}_{\ell\geq 1}$ is uniformly integrable, see e.g.~\cite[Corollary 6.21]{Kle14}, and thus~\eqref{E:measure_L1_conv} holds. %\pat{(how does convergence in the space of finite measures hold?)} 
   
   \medskip
   
   %We call this limit $\nu^\gamma(u)$ and note first that, as measure, $\nu^\gamma(B)=\nu^\gamma (\mathbf{1}_B)$ for all $B\subseteq K$ is non-trivial $\mathbb{P}$-a.s., see Corollary~\ref{C: support}. 
   
   Moreover, for each $\ell\geq 1$ and in view of~\eqref{E:def_approx_LQM}, the measure $\nu^\gamma_\ell$ is absolutely continuous with respect to $\mu$ and has Radon-Nikodym derivative $e^{\gamma h_\ell(x)-\frac{\gamma^2}2k_{d_S,\ell}(x,x)}$ for $\mu$-a.e. $x\in K$. 
   By the Cameron-Martin theorem, the Radon-Nikodym derivative is precisely the density of the law of the shifted field $h_\ell+\gamma k_{d_S,\ell}(x,\cdot)$ and thus
   \begin{equation*}
   \mathbb E\left[\int_K F(h_\ell,x)\, d\nu_{\ell}^\gamma(x)\right]=\mathbb E\left[\int_KF(h_\ell+\gamma k_{d_S,\ell}(x,\cdot),x)\, d\mu(x)\right]
   \end{equation*}
   for any positive measurable functional $F$ on $H^{-\eps}(K)\times K$. %\eva{The Cameron-Martin space is here the Sobolev space $H^{d_S/2}$. Then it holds that $\|k_\ell(x,\cdot)\|^2=k_\ell(x,x)$ and $\langle h_\ell,\gamma k_\ell(x,\cdot)\rangle=\gamma h_\ell(x)$. Thus
   %\begin{align*}
   %&\mathbb E\left[\int_K F(h_\ell,x)\, d\nu_{\ell}^\gamma(x)\right]=\int_K\mathbb E\left[ F(h_\ell,x)e^{\gamma h_\ell(x)-\frac{\gamma^2}2k_{d_S,\ell}(x,x)}\right]\, d\mu(x)\\
   %=&\mathbb E\left[\int_KF(h_\ell,x)e^{\langle h_\ell,\gamma k_\ell(x,\cdot)\rangle-\frac{\gamma^2}2\|k_\ell(x,\cdot)\|^2}\, d\mu(x)\right]
  %=\int_K \mathbb E\left[F(h_\ell+\gamma k_{d_S,\ell}(x,\cdot),x)\right]\, d\mu(x).    
   %\end{align*}
   %}
   In particular, see~\cite[Theorem 4]{Sha16}, each $\nu^\gamma_\ell$ is a subcritical GMC with expectation $\mu$. Since in addition the sequence $\{\nu^\gamma_\ell(K)\}_{\ell\geq 1}$ is uniformly integrable, it follows from~\cite[Theorem 3]{Sha16} that $\nu^\gamma$ is also a GMC, which by~\cite[Theorem 4]{Sha16} satisfies~\eqref{eq: trans}.
   %Finally, since $\{\nu_{{\ell}}^\gamma\}_{\ell\geq 1}$ converges in the space of finite measures on $K$ to $\nu^\gamma$, then~\eqref{eq: trans} holds, see~\cite[Lemma 2.1]{Aru20}. 
    %
    % Since the event $\{\widetilde{\nu}^\gamma(u)=0\}$ is independent of $h_{K,\ell}$ for any $\ell\in\mathbb N$, Kolmogorov's zero-one law implies that is has probability $0$ or $1$. By the $L^2$-convergence we know that it has positive probability and thus the limit is non-zero a.s.
    %
    %By Corollary 5 in~\cite{Sha16} it coincides with the LQGM in Proposition \ref{prop: GMC}.
\end{proof}

\begin{remark}
    In the proof of Theorem \ref{T: LQGM} we rely on the fact that $\{\nu_\ell^\gamma(K)\}_{\ell\geq 1}$ is an $L^2$-bounded martingale for $|\gamma|<\sqrt{d_H/C_0}$. Alternatively, one could prove $\sup_\ell\E[\nu_\ell^\gamma(K)^p]<\infty$ for some $p>1$ depending on $\gamma$. This would imply uniform integrability and hence convergence by Doob's convergence theorem. Such a moment estimate could potentially extend the domain of admissible values for the parameter $\gamma$. A standard way of proving $L^p$-bounds is by controlling the contribution of so-called \emph{thick points}, see \cite{DS11,Ber17}. 
    However, these arguments require precise logarithmic asymptotics of the kernel, which we do not have in this case as the available heat kernel bounds on $K$ only yield two-sided bounds with different constants.
\end{remark}

We record some properties of the LQG measure $\nu^\gamma$. 

\begin{corollary}\label{C:expected_measure}
    Let $|\gamma|<\sqrt{d_H/C_0}$. Then $\nu^\gamma$ is $\mathbb{P}$-a.s. finite and non-trivial.
\end{corollary}
\begin{proof}
%Let $E\subseteq K$. We note that by the proof of Theorem \ref{T: LQGM}, the sequence $\{\nu_\ell^\gamma(E)\}_{\ell\geq1}$ is a uniformly integrable martingale converging against $\nu^\gamma(E)$ and in particular
%\begin{align*}
    %\E[\nu^\gamma(E)]=\lim_{\ell\to\infty}\E[\nu^\gamma_\ell(E)]=\mu(E)<\infty.
%\end{align*}
Since~\eqref{E:approx_expected_measure} implies $\E[\nu^\gamma(K)]<\infty$,  it follows $\nu^\gamma(K)<\infty$ $\P$-a.s.
%\eva{The Borel-Cantelli argument is not needed here, see also below.}\pat{ Ah sehr gut, danke :)}
%Moreover the limit $\nu^\gamma(E)$ is in $L^2(\Omega)$ and thus by Tschebychev's inequality %together with the estimate~\eqref{E:L2_bounded_approx} implies
%\begin{equation*}
 %       \mathbb{P}(\nu^\gamma(E)>n)\leq\frac{\mathbb{E}[\nu^\gamma(E)^2]}{n^2}\leq \frac{\tilde{C}}{n^2}.
        %\qquad\Rightarrow\qquad
    %\sum_{n=1}^\infty \mathbb{P}(\nu^\gamma(E)>n)<\infty.
%\end{equation*}
%In particular, the Borel-Cantelli Lemma yields for $E=K$ that $\mathbb{P}(\nu^\gamma(K)<\infty)=1$.
To show that $\nu^\gamma$ is $\mathbb{P}$-a.s. non-trivial, we first note that
\begin{align*}
\mathbb{P}(\nu^\gamma(K)>0)>0.
\end{align*}
Actually,~\eqref{E:approx_expected_measure} gives $\E[\nu^\gamma(K)]=\mu(K)>0$. Since $\nu^\gamma(K)\geq0$ $\P$-a.s., the probability of the event $\{\nu^\gamma(K)>0\}$ must be positive. The latter probability equals one due to Kolmogorov's zero-one law. We shall explain the argument in more generality in the proof of Corollary \ref{C: support} below.
% Taking $u=\mathbf{1}_E$ with $E\subseteq K$ in~\eqref{E:expectation_approx_measure} (where now equality holds since $u\geq 0$) yields~\eqref{E:approx_expected_measure}.
\end{proof}

In fact, one can show that $\nu^\gamma$ has $\mathbb{P}$-a.s. full support. The proof uses Kolmogorov's zero-one law, which is a standard argument in Gaussian multiplicative chaos, see e.g. \cite[Section 2]{RV14}.
\begin{corollary}\label{C: support}
	For any $|\gamma|<\sqrt{d_H/C_0}$, the measure $\nu^\gamma$ has $\mathbb{P}$-a.s. full support.
\end{corollary}
\begin{proof}
    We first note that~\eqref{E:approx_expected_measure} implies $\mathbb{P}(\nu^\gamma(A)>0)>0$ for every non-empty open set $A\subseteq K$ since $\mu(A)>0$ and $\nu^\gamma$ is non-negative $\mathbb{P}$-a.s.
    %Since $\nu^\gamma(A)\geq0$ $\P$-a.s., the probability of the event $\{\nu^\gamma(A)>0\}$ must be positive.
    %\eva{For Patricia: If $\P[\nu^\gamma(A)>0]=0$, then $\nu^\gamma(A)=0$ $\P$-almost surely. But the expectation $\E[\nu^\gamma(A)]$ would then be zero.}\pat{Thank you got it :)}
   % Suppose that otherwise ${\rm supp}\,\nu^\gamma\subsetneq K$, in which case, there exists an open set $A\subseteq K{\setminus}{\rm supp}\,\nu^\gamma$ and $\mathbb{P}(\nu^\gamma(A)=0)=1$. Since $\nu^\gamma (A) \geq 0$ $\mathbb{P}$-a.s., the latter implies $\mathbb{E}[\nu^\gamma(A)]=0$, which by~\eqref{E:expectation_approx_measure} contradicts $\mu(A)>0$. Thus, $\mathbb{P}(\nu^\gamma(A)>0)>0$ for all open $A\subseteq K$. 
    Now, let us define for any $1\leq \ell\leq m$ the random measure
    \begin{equation*}
        d\mu_{m,\ell}^\gamma(x):=Y_{m,\ell}(x)\, dx,
    \end{equation*}
    where
    \begin{equation*}
        Y_{m,\ell}(x):=e^{\gamma(h_{m}-h_{\ell})(x)-\frac{\gamma^2}2(k_{d_S,m}-k_{d_S,\ell})(x,x)}.
    \end{equation*}
    By virtue of Theorem~\ref{T: LQGM}, for each fixed $\ell\geq 1$, the sequence $\{\mu_{m,\ell}^\gamma\}_{m\geq \ell}$ converges to a random measure $\mu_\ell^\gamma$ as $m\to\infty$ and 
    %\pat{(I need to read this again)}
    \begin{align*}
        \inf_{x\in A}e^{\gamma h_\ell(x)-\frac{\gamma^2}2k_{d_s,\ell}(x,x)}\mu_{\ell}^\gamma(A)\leq \nu^\gamma(A)\leq \sup_{x\in A}e^{\gamma h_\ell(x)-\frac{\gamma^2}2k_{d_s,\ell}(x,x)}\mu_{\ell}^\gamma(A).
    \end{align*}
    In particular, $\{\nu^\gamma(A)>0\}=\{\mu_\ell^\gamma(A)>0\}$ for all $\ell\geq 1$. Moreover, since
    \begin{equation*}
    (h_m-h_\ell)(x)=\sum_{k=N_{\ell+1}}^{N_m}\xi_{k}\lambda_{k}^{-\frac{d_S}{4}}\varphi_k(x)
    \end{equation*}
    for every $m\geq \ell$, the random measure $\mu_\ell^\gamma$ only depends on the random variables
    $(\xi_{N_{\ell+1}},\xi_{N_{\ell+2}},\ldots)$. Thus,
    \begin{align*}
    \{\mu_\ell^\gamma(A)>0\}\in
    \sigma(\xi_{N_{\ell+1}},\xi_{N_{\ell+2}},\ldots)
    \end{align*}
    and hence 
    \begin{align*}
    \{\nu^\gamma(A)>0\}\in
    \sigma(\xi_{N_{\ell+1}},\xi_{N_{\ell+2}},\ldots)
    \end{align*}
    for all $\ell\geq 1$. Consequently,
    \begin{align*}
    \{\nu^\gamma(A)>0\}\in
    \bigcap_{\ell\geq1}
    \sigma(\xi_{N_{\ell+1}},\xi_{N_{\ell+2}},\ldots),
    \end{align*}
    which is the tail sigma-field of the independent sequence $\{\xi_k\}_{k\geq1}$. By Kolmogorov's zero-one law,
    \begin{equation*}
    \mathbb{P}(\nu^\gamma(A)>0)\in\{0,1\}
    \end{equation*}
    and since we already know $\P(\nu^\gamma(A)>0)>0$, the above probability must equal one. Thus, $\nu^\gamma$ has full support $\P$-a.s. as claimed.
\end{proof}
% \begin{proof}
%     We first prove that $\nu^\gamma$ has full support with positive probability. Otherwise, there would exist an open set $O\subset K$ with $\mu(O)>0$ and $\mathbb{P}(\nu^\gamma(O)=0)=1$. Since $\nu^\gamma(O)\geq 0$ $\mathbb{P}$-a.s., the latter would imply that $\mathbb{E}[\nu^\gamma(O)]=0$ and contradict $\mu(O)>0$ because $\mathbb{E}[\nu^\gamma(O)]=\mu(O)$ by Corollary~\ref{C:expected_measure}. Applying Kolmogorov's zero-one law...\pat{(finish)}
% \end{proof}

% The second direct consequence of Proposition~\ref{P:measure_bdd_martingale}, see e.g.~\cite[Corollary 6.21]{Kle14}, is uniformly integrability.
% \begin{corollary}\label{C:approx_unif_integ}
%     Let $|\gamma|<\sqrt{d_H/C}$. Then, for any $u\in C_b(K)$, the family $\{\nu^\gamma_{h_{K,\ell}}(u)\}_{\ell\geq 1}$ is $\mathbb{P}$-a.s. uniformly integrable.
% \end{corollary}

 The next corollary in particular implies that $\nu^\gamma$ has no atoms. 
   
   \begin{corollary}
   Let $|\gamma|<\sqrt{d_H/C_0}$. Then, for any $a<d_H-C_0\gamma^2$,
   \begin{equation*}
   \mathbb{E}\left[\int_K\int_K \frac{1}{d(x,y)^a}\, d\nu^\gamma(x)\, d\nu^\gamma(y)\right]<\infty.
   \end{equation*}
   In particular, any $E\subseteq K$ with $\dim_H(E)\leq d_H-C\gamma^2$ satisfies $\nu^\gamma(E)=0$ $\mathbb{P}$-a.s.
   \end{corollary}
   \begin{proof}
   By the same arguments as in the proof of Proposition~\ref{P:measure_bdd_martingale},
   \begin{equation*}
   \sup_{\ell\in\mathbb{N}}\mathbb{E}\left[\int_K\int_K \frac1{d(x,y)^a}\, d\nu^\gamma_\ell(x)\, d\nu^\gamma_{\ell}(y)\right]\leq \int_K\int_K \frac1{d(x,y)^{a+\gamma^2C}}\, d\mu(x)\, d\mu(y),
   \end{equation*}
   where the right-hand side is finite due to Lemma~\ref{lemma: integrability}. Then, Portmanteau's theorem implies
   \begin{equation*}
    \mathbb{E}\left[\int_K\int_K \frac1{d(x,y)^a}\, d\nu^\gamma(x)\, d\nu^\gamma(y)\right]\leq\liminf_{\ell\to\infty}\mathbb{E}\left[\int_K\int_K \frac{1}{d(x,y)^a}\, d\nu^{\gamma}_{\ell}(x)\, d\nu^\gamma_{\ell}(y)\right]<\infty,
   \end{equation*}
   which yields
   \begin{equation*}
       \int_K\int_K \frac1{d(x,y)^a}\, d\nu^\gamma(x)\, d\nu^\gamma(y)<\infty
       \qquad\mathbb{P}\text{-a.s.}
   \end{equation*}   
%   \eva{Again Borel-Cantelli is not needed: Let $X$ be a nonnegative random variable such that $\E[X]<\infty$. Then $\P[X<\infty]=1$. This can be seen by contradiction: Suppose $\P[X<\infty]<1$. Then $\P[X=\infty]>0$. But then $\E[X]=\infty$.}
  Finally, Frostman's lemma~\cite[Theorem 4.13]{Fal03} implies that $\nu^\gamma(E)=0$ for any $E\subseteq K$ such that the Hausdorff dimension is at most $d_H-C_0\gamma^2$.
   \end{proof}

%%%%--------------------------------------------------------------
%%%%--------------------------------------------------------------
\subsection{Invariance properties}

\subsubsection{Invariance by symmetries}
We analyze the interplay of the LQG measure on $K$ and the symmetries $\sigma_1$, $\sigma_2$ and $\sigma_3$ described in Section~\ref{SS:symmetries_SG}. In the sequel, given a measure $\nu$ and a map $F$ on $K$ we denote by $F_\#\nu$ the push-forward of $\nu$ by $F$.

\begin{theorem}%\label{T:measure_invariance_symmetry}
    Let $|\gamma|<\sqrt{d_H/C_0}$ and let $\nu^\gamma$ be the LQG measure on $K$. For any $i=1,2,3$,
    \begin{equation*}
        \nu^\gamma\stackrel{d}{=}(\sigma_i)_\#\nu^\gamma.
    \end{equation*}
\end{theorem}
\begin{proof}
    Note that $\nu^\gamma$ is the LQG measure associated to the FGFF $h$, while $(\sigma_i)_\#\nu^\gamma$ is the LQG measure associated to $h\circ(\sigma_i)^{-1}$. Since $h\stackrel{d}{=}h\circ(\sigma_i)^{-1}$ by Theorem~\ref{T:invariance_symmetry}, the assertion follows.
\end{proof}
\subsubsection{Self-similarity by scaling}
For each word $w\in\{1,2,3\}^n$, $n\geq 1$, let $\nu_{w}^\gamma$ denote the LQG measure associated with the FGFF $h^w$ on $K^w$ defined in Theorem~\ref{T:scaling_invariance}. We show next that, in distribution, the measure $\nu_{w}^\gamma$ only depends of the length of $w$.

\begin{theorem}%\label{T:measure_invariance_scaling}
Let $|\gamma|< \sqrt{d_H/C_0}$. For any $w\in\{1,2,3\}^n$, $n\geq 1$, the LQG measures $\nu^\gamma$ and $\nu_{w}^\gamma$ satisfy
\begin{equation}\label{E:measure_invariance_scaling}
 \nu_{w}^\gamma\overset{d}{=}3^{-n}(F_w)_\#\nu^\gamma.
\end{equation}
\end{theorem}

\begin{proof}
Let $w\in\{1,2,3\}^n$, $n\geq 1$. According to Theorem~\ref{T:scaling_invariance},
$h_{w}\overset{d}{=}h\circ(F_w^{-1})$. 
Further, set
\begin{equation}\label{eq: hl}
h_{\ell}^w:=h_{\ell}\circ(F_w^{-1}),
\end{equation}
where $h_{\ell}$ is the approximate FGFF from~\eqref{E:def_approx_RF_SG}. 
By virtue of Theorem~\ref{T:convergence}, $\{h_{\ell}\}_{\ell\geq 1}$ converges $\mathbb{P}$-a.s. and in $L^2(\Omega)$ to the FGFF $h$ on $K$, and analogously $\{h^w_{\ell}\}_{\ell\geq 1}$ converges to $h^{w}$. 
For each $\ell\geq 1$, the covariance kernel associated with $h^w_{\ell}$ is given by 
\begin{equation}\label{eq: kl}
k^w_{\ell}(\tilde x,\tilde y):=\mathbb E[h_{\ell}^w(\tilde x)h_{\ell}^w(\tilde y)]=k_{d_s,\ell}(F_w^{-1}(\tilde x),F_w^{-1}(\tilde y))
\end{equation}
for all $\tilde x,\tilde y\in K_w$. 
Analogous arguments as those used in Theorem~\ref{T: LQGM} imply that the associated LQG measure $\nu_{w,\ell}^\gamma$ converges to the LQG measure $\nu_w^\gamma$ on $K_w$, i.e.
\begin{equation*}
\int_{K_w}u\, d\nu_{{w,\ell}}^\gamma\xrightarrow{L^1(\Omega)}\int_{K_w} u\, d\nu_{w}^\gamma
\end{equation*}
holds for all $u\in C_b(K_w)$.

\medskip

We now show the invariance~\eqref{E:measure_invariance_scaling}: For any $u\in C_b(K_w)$, 
%be a continuous and bounded function on $K_w$. 
it follows from~\eqref{eq: measureinv},~\eqref{eq: hl} and~\eqref{eq: kl} that
\begin{align*}
\int_{K_w}u(\tilde x)d\nu_{w}^\gamma(\tilde x)
&=\lim_{\ell\to\infty}\int_{K_w}u(\tilde x)e^{\gamma h_{\ell}^w(\tilde x)-\frac{\gamma^2}2k_{d_s,\ell}^w(\tilde x,\tilde x)}\, d\mu_w(\tilde x)\\
&=3^{-n}\lim_{\ell\to\infty}\int_{K}u(F_w(x))e^{\gamma h_{\ell}^w(F_w(x))-\frac{\gamma^2}2k_{d_s,\ell}^w(F_w(x),F_w( x))}\, d\mu(x)\\
%\int_{K_w}u(\tilde x)d\nu_{w}^\gamma(\tilde x)
&\overset{d}{=}3^{-n}\lim_{n\to\infty}\int_{K}u(F_w(x))e^{\gamma h_\ell(x)-\frac{\gamma^2}2k_{d_s,\ell}(x,x)}\, d\mu( x)\\
&=3^{-n}\int_{K}u(F_w(x))\, d\nu^\gamma(x)
\end{align*}
as we wanted to prove.
\end{proof}

%In view of the latter theorem, the Liouville quantum gravity on $K$ is defined as the Gaussian multiplicative chaos 
%\begin{equation}
   % \nu_\gamma(B):=\text{a.s.}\lim_{\ell\to\infty}\nu_{\ell,\gamma}(B),\qquad A\in\mathcal{B}(K).
%\end{equation}
%It remains to justify that this limit is non-degenerate. Here is where the $\log$-expression is necessary to apply Kahane's theorem 
%%%%%-----------------------------------------------------
%%%%%-----------------------------------------------------

\section{Liouville Brownian motion}\label{S:LBM}
Similarly to the case of planar Liouville Brownian motion (LBM) introduced by Garban-Rhodes-Vargas in~\cite{GRV16}, LBM on the Sierpinski gasket will arise as a process defined through a time change of the (standard) Brownian motion associated with the Dirichlet form~\eqref{E:def_DF_SG}.

\medskip

More precisely, let $\rho$ be an initial probability distribution on $K$ and let $\{B_t\}_{t\geq 0}$ denote the Hunt process on $(\Omega,\mathbb{Q}_\rho)$ associated with the standard Dirichlet form $(\mathcal{E},\mathcal{F})$ from~\eqref{E:def_DF_SG}. %We assume that the process and the FGFF on $K$ are independent, so that the joint law of $\{B_t\}_{t\geq 0}$ and $h_K$ is $\mathbb{Q}_\rho\otimes\mathbb{P}$. 
%\pat{(check die Richtung stimmt: unabhängig daher $\otimes$ Produkt)} - es ist dasselbe

\medskip

The LBM on $K$ constructed in this section may be regarded as the diffusion process $\{B_t^\gamma\}_{t\geq 0}$ under the LQG measure $\nu^\gamma$ analyzed in Section~\ref{SS:LQGM}. Formally, the process arises as the solution to the SDE 
\begin{equation*}
    dB_t^\gamma=e^{-\frac{\gamma}{2}h(B^\gamma_t)+\frac{\gamma^2}{4}\mathbb{E}[h(B^\gamma_t)^2]}dB_t.
\end{equation*}

By Dubins-Schwarz%\pat{(Referenz FMI? - Sebastian/Kajino)}
, the associated time change is given by 
\begin{equation*}
    A_t^\gamma=\int_0^t e^{\gamma h(B_s)-\frac{\gamma^2}{2}\mathbb{E}[h(B_s)^2]}ds.
\end{equation*}
The latter is a positive continuous additive functional (PCAF) which will be shown to be in \emph{Revuz correspondence} with the measure $\nu^\gamma$, i.e.
\begin{equation*}%\label{E:Revuz}
    \mathbb{E}_\rho\Big[\int_0^tu(B_s^\gamma)\,dA_s^\gamma\Big]=\int_0^t\int_K\int_Ku(y)p_s(x,y)d\rho(x)\,d\nu^\gamma(y)\,ds,
\end{equation*}
c.f. Theorem~\ref{T:def_LBM_SG} and~\cite[Section 5.1]{FOT11} for definitions and details concerning PCAFs and smooth measures. 

\medskip

\subsection{Capacities and smooth measures}%\label{SS:smooth_measures}
To prove the existence of the process $\{B_t^\gamma\}_{t\geq 0}$ we rely on general Dirichlet form theory and first show that the LQG measure $\nu^\gamma$ is $\mathbb{P}$-a.s.
\emph{smooth}, i.e. it is a positive Borel measure that charges no sets of $\mathcal{E}_1$-capacity zero and admits a generalized nest, see~\cite[Chapter 2.2]{FOT11}. The latter property follows from the fact that $K$ is compact and $\nu^\gamma$ is a finite measure. %\eva{Generalized nest means there exists an increasing sequence $\{F_n\}$ of closed sets such that $\nu^\gamma(F_n)<\infty$ for all $n\in\N$ and $\lim_{n\to\infty} {\rm cap}(C\setminus F_n)=0$ for any compact set $C\subseteq K$. So we could simply take $F_n=K$ for all $n\in \N$. This is an increasing sequence and since $K$ is compact they are particularly closed and we know that $\nu^\gamma(F_n)<\infty$. Moreover for every compact $C\subseteq K$ we have ${\rm cap}(C\setminus F_n)=0$. The latter follows from $C\setminus F_n=\emptyset$ and taking $u=0$ as a competitor in the definition of ${\rm cap}$.} 
Thus, it only remains to verify whether it charges no sets of $\mathcal{E}_1$-capacity zero.

\medskip

To this end, recall e.g. from~\cite[Section 2.1]{FOT11} the symmetric form $(\mathcal E_1,\mathcal{F})$ on $L^2(K)$ given by
\begin{equation*}
    \mathcal E_1(u,v):=\mathcal E(u,v)+\langle u,v\rangle
\end{equation*}
for any $u,v\in \mathcal{F}$. 
The $\mathcal E_1$-capacity of an open set $A\subset K$ is now defined as
\begin{equation}\label{E:def_cap}
    \mathrm{cap}(A):=\inf\{\mathcal E_1(u,u): u\in\mathcal F,\ u\geq 1\ \mu\text{-a.e. on }A\},
\end{equation}
where we set $\inf\emptyset:=\infty$. The $\mathcal{E}_1$-capacity of a generic set $E\subset K$ is then given by
\begin{equation*}
    \mathrm{cap}(E):=\inf_{B\text{ open, }E\subset B}\mathrm{cap}(B),
\end{equation*}
and a set $E$ with $\mathrm{cap}(E)=0$ is called an $\mathcal E_1$-polar set.

\begin{remark}
	Here, we use the convention of $\mathcal E_1$-capacity, which is used in~\cite[Section 2.1]{FOT11}. 
    Since $\mathcal E(u,u) \geq\lambda_1\|u\|_{L^2}^2$ due to the positivity of the first Dirichlet eigenvalue of the Laplacian, $\mathcal E$ and $\mathcal E_1$ determine equivalent norms on $\mathcal F$. 
    In particular, $\mathcal E$-polar sets coincide with $\mathcal E_1$-polar sets.
\end{remark}

In order to prove that the measure $\nu^\gamma$ charges no $\mathcal E_1$-polar sets, i.e. that for $E\subset K$
\begin{align*}
    \mathrm{cap}(E)=0\quad \Rightarrow\quad \nu^\gamma(E)=0~\quad\mathbb{P}\text{-a.s.},
\end{align*}
we study the related \emph{Bessel potentials}.
By analogy with~\cite[Definition 2.3]{HZ05}, for $s>0$ and $\alpha\geq0$ we define the Bessel kernel
\begin{align*}
r_{s,\alpha}(x,y):=\frac1{\Gamma(s)}\int_0^\infty e^{-\alpha t}t^{s-1}p_t(x,y)\, dt, \quad x,y\in K,\ x\neq y,
\end{align*}
and its associated Bessel potential
\begin{equation*}
R_{s,\alpha}f(x):=\int_K r_{s,\alpha}(x,y)f(y)\, d\mu(y),
\end{equation*}
where $f\in L^2(K)$. 
Similarly to~\cite[Proposition 2.4]{HZ05} %\cite[Lemma 2.9]{BC23}
we may write the latter in terms of the infinitesimal generator $\Delta$ as
\begin{equation*}
    R_{s,\alpha}f=(\alpha I-\Delta)^{-s}f,
\end{equation*} 
for any $f\in L^2(K)$. 
%Moreover, for all $\alpha>0$ the kernel satisfies the convolution property
%\begin{align}\label{eq: conv}
 %   r_{s,\alpha}\ast r_{t,\alpha}=r_{s+t,\alpha}, \quad s,t>0,
%\end{align}
%cf. \cite[Corollary 2.11]{BC23}.
For $\nu\in\mathcal M_b(K)$, we further set
\begin{equation*}
R_{s,\alpha}\nu(x):=\int_K r_{s,\alpha}(x,y)\, d\nu(y).
\end{equation*}

\begin{lemma}\label{lemma: ziemer}
For any $E\subseteq K$, let
    \begin{equation*}
        \begin{aligned}
        b(E)
        := \sup\{\nu(E)\colon \nu\in\mathcal{M}_b(K),\|R_{1/2,1}(\mathbf{1}_E\nu)\|_{L^2}\leq 1\}.
        \end{aligned}
    \end{equation*}
Then, for every open set $A\subset K$ it holds that
    \begin{equation}\label{E:cap_is_b2}
    \mathrm{cap}(A)= b(A)^2.
    \end{equation}
\end{lemma}
\begin{proof}
    We first show that
    \begin{equation}\label{E:equiv_cap_Ziemer}
      \mathrm{cap}(A)=  \inf\{\|f\|_{L^2}^2\colon R_{1/2,1}f\geq1\ \mu\text{-a.e. on }A,\ f\geq 0\}.
    \end{equation}
    Afterwards, the claim~\eqref{E:cap_is_b2} will follow from~\cite[Theorem 2.6.12]{Zie89}. {Note that~\cite[Theorem 2.6.12]{Zie89} is formulated for Euclidean Bessel potentials, but its proof is purely variational and can thus be adapted to our setting. See also the analogous adaptation in \cite[Lemma 5.5]{dSHKS24}.} 
    In order to show~\eqref{E:equiv_cap_Ziemer}, note that the spectral theorem yields
\begin{equation}\label{E:inverse_E1}
    R_{1/2,1}f=\sum_{k= 1}^\infty\frac1{(1+\lambda_k)^{1/2}}\langle f,\varphi_k\rangle \varphi_k
\end{equation}
for any $f\in L^2(K)$, and also $R_{1/2,1}f\in \mathcal F$ since
\begin{align*}
    \mathcal E_1(R_{1/2,1}f,R_{1/2,1}f)
    =\sum_{k=1}^\infty \frac{1+\lambda_k}{1+\lambda_k}\langle f,\varphi_k\rangle^2=\|f\|^2_{L^2}.
\end{align*}
Thus, for any $f\in L^2(K)$ with $R_{1/2,1}f\geq 1$ on $A$, the function $u=R_{1/2,1}f$ belongs to the family over which the capacity~\eqref{E:def_cap} minimizes, so that 
\begin{align*}
    \mathrm{cap}(A)\leq \inf\{\|f\|_{L^2}^2\colon R_{1/2,1}f\geq1\ \mu\text{-a.e. on }A,\ f\geq 0\}.
\end{align*}
This shows one inequality in~\eqref{E:equiv_cap_Ziemer}. To prove the converse, for $u\in \mathcal F$ such that $u\geq 1$ $\mu$-a.e. on $A$, set $f:=g_+$, where $g:=(I-\Delta)^{1/2}u$. Then, $f\in L^2(K)$ and $f\geq0$. In addition, since $f\geq g$ and the kernel $r_{1/2,1}$ is nonnegative, it follows that
\begin{equation*}
R_{1/2,1}f\geq R_{1/2,1}g=u,
\end{equation*}
and from~\eqref{E:inverse_E1}
\begin{align*}
    \mathcal E_1(u,u)=\|g\|^2_{L^2}\geq \|f\|_{L^2}^2.
\end{align*}
Thus, any $u$ in the infimum~\eqref{E:def_cap} yields a function $f$ that belongs to the family on the RHS of~\eqref{E:equiv_cap_Ziemer}, whence
\begin{equation*}
    \mathrm{cap}(A)\geq  \inf\{\|f\|_{L^2}^2\colon R_{1/2,1}f\geq1\ \mu\text{-a.e. on }A,\ f\geq 0\},
\end{equation*}
proving the equality~\eqref{E:equiv_cap_Ziemer}.
\end{proof}

With the help of the previous lemma we now obtain the smoothness of the measure $\nu^\gamma$.

\begin{lemma}%\label{L:smooth_measure}
    Let $|\gamma|< \sqrt{d_H/C_0}$. Then, $\P$-a.s. the LQG measure $\nu^\gamma$ charges no $\mathcal E_1$-polar sets.
\end{lemma}

\begin{proof}
%We first note that $\P$-a.s.
%\begin{align}\label{eq: sup}
 %   \sup_{x\in K}R_{1,1}\nu^\gamma(x)%\leq C\nu^\gamma(K)
  %  <\infty.
%\end{align}
%Indeed, since $r_{1,1}\leq r_{1,0}$ it follows
%\begin{align*}
%R_{1,1}\nu^\gamma(x)\leq \int_K\int_0^\infty r_{1,0}(x,y)	\, d\nu^\gamma(y).
%\end{align*}
%Proposition 2.8 in \cite{BC23} and Corollary \ref{C:expected_measure} imply \eqref{eq: sup}.
   
First, we show that 
\begin{equation}\label{eq: l2est}
    \|R_{1/2,1}(\nu^\gamma)\|_{L^2}^2<\infty.
\end{equation}
Since $r_{1/2,1}\leq r_{1/2,0}$, the convolution property of these kernels, see~\cite[Corollary 2.11]{BC23}, implies
\begin{align*}
\|R_{1/2,1}(\nu^\gamma)\|_{L^2}^2=&\int_K\int_K\int_K r_{1/2,1}(x,y)r_{1/2,1}(x,z)\, d\nu^\gamma(y)\, d\nu^\gamma(z)\, d\mu(x)\\
\leq &\int_K\int_K\int_K r_{1/2,0}(x,y)r_{1/2,0}(x,z)\, d\nu^\gamma(y)\, d\nu^\gamma(z)\, d\mu(x)\\
=&\int_K\int_K\int_K r_{1,0}(y,z)\, d\nu^\gamma(y)\, d\nu^\gamma(z).
\end{align*}
Further, note that $r_{1,0}(x,y)$ is uniformly bounded in $x,y\in K$ by virtue of~\cite[Proposition 2.8]{BC23} and consequently
the last integral above is indeed $\P$-a.s. finite since $\nu^\gamma(K)<\infty$ $\P$-a.s. by Corollary \ref{C:expected_measure}.

\medskip

 Next, let $\eps>0$ and let $E\subset K$ be a $\mathcal E_1$-polar set. By definition, there exists an open set $A\subset K$ such that $E\subset A$ and $\mathrm{cap}(A)< \eps$.
 By applying the definition of $b(A)$ to $\tilde \nu:=\|R_{1/2,1}(\mathbf{1}_A\nu^\gamma)\|_{L^2}^{-1}1_A\nu^\gamma$ we obtain 
 \begin{equation*}
     \nu^\gamma(E)\leq \nu^\gamma(A)\leq b(A)\|R_{1/2,1}(\mathbf{1}_A\nu^\gamma)\|_{L^2}\qquad \P\text{-a.s.}
 \end{equation*}
 Further, by Lemma~\ref{lemma: ziemer} and since $R_{1/2,1}(\mathbf{1}_A\nu^\gamma)\leq R_{1/2,1}(\nu^\gamma)$
 \begin{equation*}
     \nu^\gamma(E)\leq\mathrm{cap}(A)^{1/2}\|R_{1/2,1}(\nu^\gamma)\|_{L^2}< \|R_{1/2,1}(\nu^\gamma)\|_{L^2}\eps^{1/2}.
 \end{equation*}
Now, the claim follows from~\eqref{eq: l2est} because $\eps>0$ was arbitrary chosen. 
\end{proof}

\begin{corollary}\label{C:nu_smooth}
    Let $|\gamma|< \sqrt{d_H/C_0}$. Then, the LQG measure $\nu^\gamma$ on $K$ is $\P$-a.s. smooth.
\end{corollary}
\subsection{LBM on the Sierpinski gasket}%\label{SS:def_LBM_SG}
The one-to-one correspondence between smooth measures and PCAFs, c.f.~\cite[Theorem 5.1.4]{FOT11}, guarantees the existence up to equivalence of PCAFs of the time changed process $\{B_t^\gamma\}_{t\geq 0}$ on $K$ given by
%\pat{(check this def from FOT p.406 vs $B_{A_t^\gamma}$)} --- OK this is the one to take :)
\begin{equation}\label{E:def_LBM_SG}
        B_t^\gamma:=B_{\tau_t^\gamma},\qquad\text{where}\quad\tau_t^\gamma:=\inf\{s>0\colon A_t^\gamma>s\}.
\end{equation}
The main result of this section collects relevant features of this process, which we shall call the \emph{Liouville Brownian Motion} (LBM) on $K$. In terms of notation, $(\mathcal{E},\mathcal{F}_e)$ below refers to the the \emph{extended Dirichlet space} of the Dirichlet form $(\mathcal{E},\mathcal{F})$, c.f.~\cite[p.41]{FOT11}.
%%%--------
\begin{theorem}\label{T:def_LBM_SG}
    Let $|\gamma|<\sqrt{d_H/C_0}$. Then, $\mathbb{P}$-a.s. it holds that:
    \begin{enumerate}[wide=0em,itemsep=.25em,label={\rm(\roman*)}]
        \item There exists a unique PCAF $\{A_t^\gamma\}_{t\geq 0}$ with Revuz measure $\nu^\gamma$.
        \item The time-changed process $\{B_t^\gamma\}_{t\geq 0}$ given by~\eqref{E:def_LBM_SG} is a $\nu^\gamma$-symmetric Hunt process with associated Dirichlet form
        \begin{equation}\label{E:time_changed_DF}
            \mathcal{E}^\gamma(u,v)=\mathcal{E}(u,v),\qquad\mathcal{F}^\gamma=\mathcal{F}_e\cap L^2(K,\nu^\gamma).
        \end{equation}       
    \end{enumerate}
\end{theorem}

\begin{proof}[Proof of Theorem~\ref{T:def_LBM_SG}]
    \begin{enumerate}[wide=0em,itemsep=.25em,label=(\roman*)]
        \item The existence of the PCAF $\{A_t^\gamma\}_{t\geq 1}$ follows from~\cite[Theorem 5.1.4]{FOT11} and the smoothness of $\nu^\gamma$ from Corollary~\ref{C:nu_smooth}.
        \item By virtue of~\cite[Theorem 6.2.1]{FOT11}, the time-changed process~\eqref{E:def_LBM_SG} is a (symmetric) Hunt process and the associated Dirichlet form takes the expression~\eqref{E:time_changed_DF} by virtue of~\cite[(6.2.22)]{FOT11} since $\nu^\gamma$ has full support $\mathbb{P}$-a.s. by Corollary~\ref{C:nu_smooth}.
    \end{enumerate}
\end{proof}

Finally, we mention that, if additionally $\rho$ has bounded density with respect to $\mu$, an adaptation of the arguments in the proof of \cite[Theorem 5.7]{dSHKS24} shall lead to the approximation
    \begin{equation*}
     A_t^\gamma=\lim_{\ell\to\infty}\int_0^te^{\gamma h_{\ell}(B_s)-\frac{\gamma^2}{2}k_{d_S,\ell}(B_s,B_s)}ds
    \end{equation*}
$\Q_\rho\otimes\P$-a.s. for all $t\geq0$.
%\end{proposition}
In particular, the PCAFs associated to $\nu^\gamma_\ell$ will converge to the PCAF associated to $\nu^\gamma$.
%%%%%------------------------------------------------------
%\newpage
\bibliographystyle{amsplain}
\bibliography{LQG_SG_refs}
\end{document}